\documentclass[a4paper,11pt]{amsart}
\usepackage{amsfonts}
 \usepackage{amsthm}
\usepackage{amssymb}
\usepackage{amsmath}
\usepackage{bbm}
\usepackage{enumerate}
\usepackage[pdftex]{graphicx}
\usepackage{url}
\usepackage{color}
\usepackage{a4wide}
\usepackage{mathrsfs}

\usepackage[T1]{fontenc}
\usepackage{textcomp}

\newcommand{\eps}{\varepsilon} 
\numberwithin{equation}{section}
\theoremstyle{plain}
\newtheorem{thm}{Theorem}[section]

\newtheorem{lemma}[thm]{Lemma}

\newtheorem{prop}[thm]{Proposition}
\newtheorem{conj}[thm]{Conjecture}
\newtheorem{remark}[thm]{Remark}
\newtheorem{thmalpha}{Theorem}

\theoremstyle{definition}

\newtheorem{problem}[thm]{Problem}

\theoremstyle{remark}

\newcommand{\N}{\mathbb{N}}
\newcommand{\R}{\mathbb{R}}
\newcommand{\E}{\mathbb{E}}
\newcommand{\ve}{\varepsilon}
\newcommand{\Pro}{\mathbb P}
\newcommand{\Var}{\mathrm{Var}}

\newcommand{\rate}{\mathbb I}
\newcommand{\discset}{{\mathcal D}}
\newcommand{\discsetinf}{{\mathcal D}_\infty}
\newcommand{\mapT}{{\mathcal T}}

\newcommand{\yspace}{{\mathcal Y}}
\newcommand{\yfilt}{{\mathcal A}}
\newcommand{\py}{\mathbb{P}_{\small{ {\mathcal Y}}}}

\newcommand{\term}{H}
\newcommand{\integ}{M}

\begin{document}
\title{Large Deviation Principles for Lacunary Sums}
\author[C. Aistleitner, N. Gantert, Z. Kabluchko, J. Prochno, K. Ramanan]{Christoph Aistleitner, Nina Gantert, Zakhar Kabluchko,\\ Joscha Prochno, Kavita Ramanan}

\newcommand{\mods}[1]{\,(\mathrm{mod}\,{#1})}

\begin{abstract}
  Let $(a_k)_{k\in \mathbb N}$ be an increasing sequence of positive integers satisfying the Hadamard gap condition $a_{k+1}/a_k> q >1$ for all $k\in\mathbb N$, and let 
$$
S_n(\omega) = \sum_{k=1}^n \cos(2\pi a_k \omega), \qquad n\in\mathbb N, \; \omega\in [0,1].
$$
Then $S_n$ is called a lacunary trigonometric sum, and can be viewed as a random variable defined on the probability space $\Omega= [0,1]$ endowed with Lebesgue measure. Lacunary sums are known to 
exhibit several properties that are typical for sums of independent random variables.
For example, a central limit theorem for $(S_n)_{n\in \N}$
has been obtained by Salem and Zygmund, while a law of the iterated logarithm is due to Erd\H os and G\'al.
In this paper we initiate the investigation of large deviation principles for lacunary sums.
Specifically, under the large gap condition $a_{k+1}/a_k \to\infty$, we prove that
the sequence $(S_n/n)_{n \in \N}$ does indeed satisfy a large deviation principle with speed $n$ and the same rate function $\widetilde{I}$ as for sums of independent random variables with the arcsine distribution. On the other hand, we show that the large deviation principle may fail to hold when we only assume the Hadamard gap condition.
However, we show that in the special case when $a_k= q^k$ for some $q\in \{2,3,\ldots\}$,  $(S_n/n)_{n \in \N}$ satisfies a large deviation principle (with speed $n$) and a rate function $I_q$ that is different from $\widetilde{I}$, and
describe an algorithm  to compute an arbitrary number of terms in the Taylor expansion of $I_q$.
 In addition,  we also prove that $I_q$ converges pointwise to $\widetilde I$ as $q\to\infty$.
 Furthermore, we construct a random perturbation $(a_k)_{k \in \N}$ of the sequence $(2^k)_{k \in \N}$ for which $a_{k+1}/a_k \to 2$ as $k\to\infty$, but for which at the same time $(S_n/n)_{n \in \N}$  satisfies a large deviation principle
 with the same rate function $\widetilde {I}$ as in the independent case, which is surprisingly different from the rate function $I_2$ one might na{\"i}vely expect. We relate this fact to the number of solutions of certain Diophantine equations. 
 Together, these results  show that large deviation principles for lacunary trigonometric sums are very sensitive to
 the arithmetic properties of the sequence $(a_k)_{k\in \mathbb N}$. This is particularly noteworthy since no such arithmetic effects are visible in the central limit theorem or in the law of the iterated logarithm for lacunary trigonometric sums. The proofs use a combination of tools from  probability theory,
   harmonic analysis, and  dynamical systems.
 \end{abstract}

\maketitle

\noindent
    {\em MSC2010 subject classifications.} Primary 42A55, 60F10, 11L03; 
Secondary 37A05, 11D45,  11K70.   

\allowdisplaybreaks


\section{Introduction}

The study of lacunary series is a classical and still flourishing topic in harmonic analysis that has attracted considerable attention. In the article \cite{Rademacher1922} published in 1922, Rademacher studied the convergence  behavior of series of the form 
\[
\sum_{k=1}^\infty b_kr_k(\omega), 
\]
where $\omega \in[0,1]$, $b=(b_k)_{k\in\N}\in \R^\N$, and $r_k$ denotes the $k^{\text{th}}$ Rademacher function, that is, $r_k(\omega)=\text{sign}\big( \sin(2^k\pi \omega)\big)$. He proved that such a series converges for almost every $\omega \in[0,1]$
if $\sum_{k \in \N} |b_k|^2 < \infty$,
or equivalently,  $b\in\ell_2$. The necessity of square summability was obtained shortly after by Khintchine and Kolmogorov in  1925 \cite{KK1925}, thereby establishing an interesting $\ell_2$-dichotomy in the convergence behavior of such series.
Note that by the structure of the Rademacher functions one has
$$
\sum_{k=1}^\infty b_kr_k(\omega) = \sum_{k=1}^\infty b_k r_1 (2^{k-1} \omega),
$$ 
where on the right-hand side we have a series of dilates of a fixed function,
with an exponentially growing dilation factor. This leads naturally to the study of similar questions for lacunary trigonometric series, that is, series of the form
\[
\sum_{k=1}^\infty b_k\cos(2\pi a_k \omega) \qquad \text{and} \qquad \sum_{k=1}^\infty b_k\sin(2\pi a_k \omega),
\]
where $\omega \in[0,1]$, $b=(b_k)_{k\in\N}\in \R^\N$, and $(a_k)_{k\in\N}$ is   a  sequence  of positive integers that is  lacunary, in the sense that it satisfies the Hadamard gap condition
\begin{equation}
\label{Hadamard-gap}
\frac{a_{k+1}}{a_k} \geq q>1, \qquad \mbox{ for every } k \in \N. 
\end{equation}
Interestingly, results similar to the Rademacher case were obtained for such
series. Kolmogorov showed in \cite{K1924} that the square summability of $b$ is sufficient for the almost everywhere convergence of lacunary series and Zygmund proved in \cite{Z1930} that the square summability condition
was necessary, again establishing the same $\ell_2$-dichotomy as for Rademacher series.

An important property of the Rademacher functions is that they form a system of independent random variables. More precisely,  if $(b_k)_{k\in\N}$ is a sequence of real numbers, then the weighted Rademacher functions $b_kr_k$, $k\in\N,$ form a sequence of independent and centered random variables with $\Var(b_kr_k) = b_k^2$. One readily checks that Lindeberg's condition is satisfied whenever both $b\notin\ell_2$ and $\max_{1\leq k\leq n}|b_k|=o(\|(b_k)_{k=1}^n\|_2)$. This means that under these two conditions we have, for every $t\in\R$, the central limit theorem (CLT)
\[
\lim_{n\to\infty} \lambda\bigg(\Big\{ \omega \in[0,1]\,:\, \sum_{k=1}^n b_kr_k(\omega ) \leq t \|(b_k)_{k=1}^n\|_2 \Big\} \bigg) = \frac{1}{\sqrt{2\pi}}\int_{-\infty}^t e^{-y^2/2}\,dy,
\]
which in particular holds when $b_k=1$ for every $k\in\N$.
In 1939  Kac proved an analogous central limit theorem in the lacunary case for integer 
sequences $(a_k)_{k\in\N}$ with very large gaps, that is, those for which $a_{k+1}/a_k\to\infty$, as $k\to\infty$. The general case, however, did not appear till 1947 when Salem and Zygmund established in \cite{SZ1947} that, for all $t\in\R$, 
\begin{align}\label{thm:clt salem zygmund}
\lim_{n\to\infty} \lambda\bigg(\Big\{ \omega \in[0,1]\,:\, \sum_{k=1}^n \cos(2\pi a_k \omega) \leq t \sqrt{n/2} \Big\} \bigg) = \frac{1}{\sqrt{2\pi}}\int_{-\infty}^t e^{-y^2/2}\,dy.
\end{align}
These results suggest that lacunary trigonometric sums   behave in many ways
like sums of independent random variables, and in fact, this has become a classical heuristic
that has been confirmed in many settings.  
Indeed, under the Hadamard gap condition, the sequence of scaled partial sums of the functions $\cos(2\pi a_kx)$, $k\in\N,$ not only satisfies the central limit theorem in \eqref{thm:clt salem zygmund}, but,  as Salem  and Zygmund  \cite{SZ1950}
and Erd\H os and G\'al \cite{EG1955} showed,  also satisfies a law of the iterated logarithm (LIL), that is, for almost every $\omega \in[0,1]$, 
\[
\limsup_{n\to\infty}\, \frac{\sum\limits_{k=1}^n \cos(2\pi a_k \omega)}{\sqrt{n\log\log n}} = 1.
\]
A generalization  to non-integral sequences  $(a_k)_{k \in \N}$
 was also later established in \cite{MW1959}.

A natural question is to ask whether the CLT and LIL still hold under the Hadamard gap condition
when the function $\omega \mapsto \cos(2 \pi \omega)$ is replaced by an
arbitrary $1$-periodic function $f$. 
A famous example of  Erd\H os and Fortet (see, e.g., \cite{K1949}) shows that this is not true in general.  
However, under the additional condition that the function $f:\R\to\R$ is of bounded variation on $[0,1]$ and satisfies both 
\begin{equation}
\label{Kac-conditions}
f(\omega +1) = f(\omega) \qquad\text{and}\qquad \int_0^1 f(\omega) \,d\omega = 0,
\end{equation}
Kac was able to show in \cite{K1946} that a central limit theorem holds for scaled partial sums of the functions $\omega \mapsto f(2^k\omega)$, $k\in\N,$ but in this  case the variance of the Gaussian limit law is 
\begin{equation}\label{eq:sigma^2 non-independet}
\sigma^2 = \int_0^1 f(\omega)^2\,d\omega + 2\sum_{k=1}^\infty\int_0^1 f(\omega)f(2^k\omega)\,d\omega,  
\end{equation}
rather than   $\int_0^1 f(\omega)^2\,d\omega$, as one would have 
  in the independent case, namely for the sequence of partial sums $\sum_{k=1}^n f(2^kU_k)$, where
  $\{U_k\}_{k \in \N}$ are independent and identically distributed (i.i.d.) random
  variables distributed uniformly on $(0,1)$. 
  This shows that general lacunary function systems possess a more complicated dependence
structure than lacunary trigonometric function systems, and that in the general case, the arithmetic
structure of the lacunary integer sequence plays a crucial role. 
 Gapo\v{s}kin found a remarkable relation between the existence of a CLT and the number of solutions to a certain Diophantine equation \cite{G1970}. It was only more recently,  in 2010, that Aistleitner and Berkes improved Gapo\v{s}kin's result and provided the precise condition for the central limit theorem to hold in the general lacunary framework \cite[Theorem 1.1]{AB2010}.

 While we have seen that the probabilistic behavior of lacunary series is quite well understood  on the scales of
 both the CLT and LIL, this is not the case for large deviations. Specifically,  large deviation principles (LDPs)
 seem to have not been studied at all in the lacunary setting. 
 In contrast to the CLT, which captures universal behavior in the sense that the limits are insensitive to details of the distribution beyond the first and second moments,  probabilities of (large) deviations on the scale of laws of large numbers are non-universal and describe the asymptotic likelihood of rare events. More precisely, LDPs are sensitive to the distribution of the underlying random variables and their non-universality is reflected in the so-called rate function and/or the speed, which together define the asymptotic exponential decay rate of
  large deviation probabilities. The most classical result in this direction is Cram\'er's theorem \cite{C1938} (see also \cite{CT2018} and \cite[Theorem 2.2.3]{dembo_zeit}), which guarantees that if $X,X_1,X_2,\ldots$ are i.i.d.~random variables with cumulant (or log-moment) generating function $\Lambda(u):=\log\E[e^{uX}]<\infty$ for $u$ in a neighborhood of zero,  then one has 
$$
\lim\limits_{n\to\infty}{1\over n}\log \Pro\left(X_1+\ldots+X_n\geq nt\right)=-\Lambda^*(t), 
$$
for all $t>\E [X]$, where $\Lambda^*$ is the Legendre-Fenchel transform of 
$\Lambda$ given by
$$
\Lambda^*(t) = \sup_{\theta \in \R} \left[\theta x - \Lambda(\theta)\right]\, .
$$
LDPs in the spirit of Donsker and Varadhan, who initiated a systematic study (see \cite{dembo_zeit, V2008} and the references cited therein), generalize the idea behind Cram\'er's theorem. Loosely speaking, a sequence $(X_n)_{n \in \N}$ of random variables in $\R^d$ is said to satisfy an LDP with speed $s_n\uparrow \infty$ and a  rate function $\rate:\R^d\to[0,\infty]$ if for sufficiently large $n\in\N$ and $A\subset \R^d$ sufficiently regular,
\[
\Pro(X_n\in A) \approx e^{-s_n \, \inf\limits_{x\in A}\rate(x)}.
\]
More precisely, a sequence $(X_n)_{n \in \N}$ of random variables in $\R^d$ is said to satisfy an LDP with speed $s_n$ and rate function $\rate:\R^d\to[0,\infty]$ if
$\rate:\R^d \to [0,\infty]$  is   lower-semicontinuous and 
 for every Borel measurable set $A \subset \R^d$,
 \begin{equation}
   \label{def-ldp}
   - \inf_{x \in A^\circ} \rate(x) \leq \liminf_{n\to\infty} \frac{1}{s_n} \log \Pro( X_n \in A) \leq \limsup_{n\to\infty} \frac{1}{s_n} \log \Pro( X_n \in A)
\leq - \inf_{x \in \bar{A}} \rate (x),
\end{equation}
where $A^\circ$ and $\bar{A}$, respectively, denote the interior and closure of the set $A$.

In this paper we initiate the study of large deviations for lacunary sums,
thereby complementing 
 existing  limit theorems like the CLT and LIL 
 mentioned above. More precisely, if $(a_k)_{k\in\N}$ is a lacunary sequence, that is, a sequence of real numbers satisfying the Hadamard gap condition \eqref{Hadamard-gap}, 
we study the  tail behavior 
  of the associated sequence of lacunary sums, namely partial sums of the  sequence
$X_k(\omega):=\cos(2\pi a_k \omega)$,
$\omega \in [0,1]$, $k\in\N,$ viewed as  real-valued random variables on the  space 
  $[0,1]$ equipped with the Borel $\sigma$-algebra $\mathcal B([0,1])$ and Lebesgue 
  measure $\lambda$.  
 Our results reveal an interesting and surprising behavior, showing how sometimes -- depending on arithmetic properties of the lacunary sequence $(a_k)_{k \in \N}$ -- the large deviations behavior of the associated sequence of lacunary trigonometric  sums resembles that of partial  sums of independent and identically distributed random variables, whereas in other situations it does not.  This is particularly interesting since no such influence of the arithmetic structure of the lacunary sequence  is visible under the Hadamard gap condition when considering lacunary trigonometric sums, neither in the case of the CLT nor in the case of the LIL.

 We present precise statements of our main findings in the next section, with the proofs presented
   in the following section.

\section{Main results}

We  now present the main results of this paper. 
Let $U \sim {\rm Unif}(0,1)$ be a random variable with the uniform distribution on the interval $[0,1]$.  Given a sequence $(a_k)_{k \in \N}$ of positive integers, define the random variables
    \begin{equation}
    \label{def-xk}
    X_k := \cos (2 \pi a_k U), \quad k \in \N,
    \end{equation}
    and their partial sums
    \begin{equation}
    \label{def-sn}
    S_n := \sum_{k=1}^n X_k = \sum_{k=1}^n  \cos (2 \pi a_k U),   \quad n \in \N.
    \end{equation} 
    These random variables are most conveniently defined on the probability space $\Omega=[0,1]$ endowed with the Borel $\sigma$-algebra $\mathcal B([0,1])$ and  standard Lebesgue measure $\lambda$,
    which we shall sometimes also denote by $\Pro$.    
    As a function on $\Omega$,  $X_k$ is then given by $X_k (\omega) = \cos (2\pi a_k \omega)$, for $\omega\in [0,1]$  and $k\in\N$.
    Note that the random variables $X_1,X_2,\ldots$ are identically distributed and (if all  $a_k, k \in \N,$  are distinct) uncorrelated. To see that the correlations vanish when  $(a_k)_{k \in \N}$ are distinct, recall that $\cos(\alpha)\cdot\cos(\beta) = 2^{-1}[\cos(\alpha-\beta) + \cos(\alpha + \beta)]$ and hence, whenever $k \neq \ell$, we have  
\[
\int_{0}^{1} \cos(2\pi a_k \omega) \cos(2\pi a_\ell \omega)\,d \omega = \frac{1}{2}\int_0^1 \cos(2\pi(a_k-a_\ell)\omega) \,d\omega + \frac{1}{2} \int_0^1 \cos(2\pi(a_k+a_\ell)\omega) \,d\omega = 0.
\]
However, the elements of the sequence $(X_k)_{k\in\N}$ are not independent and  in fact, the sequence is in general not even stationary.

\subsection{Behavior as in the independent case}
    Our aim is to prove LDPs for the sequence $(S_n/n)_{n\in\N}$.
    It is natural to try to compare the behavior of $S_n/n$ to the behavior of partial sums of \textit{independent} random variables with the same distribution as $X_1$, the common distribution of $X_k, k \in \N$.  To this end, consider the random variables
    \begin{equation}
      \label{def-wtildexk}
      \widetilde{X}_k := \cos (2 \pi a_k  U_k),  \quad k \in \N,
      \end{equation}
    where $(U_k)_{k \in \N}$ are i.i.d.~random variables with the same distribution as $U$, and define their partial sums
    \begin{equation}
      \label{def-wtildesn} 
    \widetilde{S}_n := \sum_{k=1}^n \widetilde{X}_k,  \quad n \in \N.
    \end{equation}
By Cram\'er's classical theorem (see, e.g., \cite[Theorem 2.2.3]{dembo_zeit}), $(\widetilde S_n/n)_{n\in\N}$ satisfies an LDP with speed $n$ and rate function $\widetilde{I}:\R\to[0,+\infty]$ given by the Legendre-Fenchel transform of the logarithmic moment generating function, that is,
\begin{equation}
\label{def-widetildeI}
\widetilde I(x) = \sup_{\theta \in \R} \left[\theta x - \widetilde \Lambda(\theta)\right],
\end{equation}
where
\begin{equation}
\label{def-tildeLambda}
\widetilde \Lambda (\theta) := \log\E[e^{\theta \widetilde{X}_1}],
\qquad
\theta\in\R.
\end{equation} 
The function  $\widetilde{\Lambda}$ can be computed explicitly. 
The common distribution of the random variables $\widetilde{X}_k, k \in \N,$ is the arcsine law on $(-1,1)$ with Lebesgue density
\[
f(x) = \frac{1}{\pi \sqrt{1-x^2}}\,, \qquad |x|<1. 
\]
The moment generating function of $\widetilde X_1$ is accordingly given by
\begin{eqnarray}\label{arcsin_mgf}
\E[ e^{\theta \widetilde{X}_1}]
=
\int_{-1}^1 e^{\theta x} \frac{1}{\pi \sqrt{1 - x^2} }dx & = & \sum_{m=0}^\infty \frac{\theta^{2m}}{(2m)!} \int_{-1}^1 \frac{x^{2m}}{\pi \sqrt{1-x^2}} dx \nonumber\\
& = & \sum_{m=0}^\infty \frac{\theta^{2m}}{(2m)!} \frac{\Gamma(m+1/2)}{\Gamma(m+1)\sqrt{\pi}} \nonumber\\
& = & \sum_{m=0}^\infty \frac{\theta^{2m}}{(2m)!} \frac{(2m)!}{2^{2m} m!m!} \nonumber\\
& = & \sum_{m=0}^\infty \frac{\theta^{2m}}{2^{2m} (m!)^2}.
\end{eqnarray}
Note that the right-hand side is equal to the modified Bessel function $B_0(\theta)$
of the first kind. 
When combined, the above calculations yield 
\begin{equation}\label{arcsin_cumgf}
\widetilde \Lambda (\theta)
=
\log \sum_{m=0}^\infty \frac{\theta^{2m}}{2^{2m} (m!)^2}
,
\qquad
\theta\in\R.
\end{equation}
Since $\widetilde X_1$ is supported on the interval $[-1,1]$, the function $\widetilde I$ equals $+\infty$ outside $[-1,1]$. Moreover, the asymptotics of the modified Bessel function $B_0$ given in~\cite[p.~377, 9.7.1]{abramowitz_stegun} imply that
$$
\widetilde \Lambda (\theta) = \theta - \frac{1}{2} \log (2\pi \theta) + O\left(\frac{1}{\theta}\right), \qquad \mbox{ as } \theta\to +\infty,
$$
which, after taking the Legendre-Fenchel transform, yields that $\widetilde I(\pm 1) = +\infty$. On the interval $(-1,1)$, the function $\widetilde I$ is finite.

\vspace*{2mm}
Now, let us finally turn to the partial sums $(S_n)_{n \in \N}$ defined in \eqref{def-sn}. Our first result states that when $(a_k)_{k  \in \N}$ satisfies the so-called ``large gap condition'', the associated sequence of lacunary sums $(S_n/n)_{n\in \N}$ satisfies an LDP  with the same speed and the same rate function $\widetilde I$ as in the truly independent case, that is, as
$(\widetilde{S}_n/n)_{n \in \N}$.

\begin{thmalpha} \label{th1}
Suppose that $(a_k)_{k\in\N}$ is a sequence of positive integers that satisfies the ``large gap condition'' 
\[
\frac{a_{k+1}}{a_k} \to \infty \quad \text{as} \quad k \to \infty.
\]
Then $(S_n/n)_{n \in\N}$ satisfies an LDP with speed $n$, and rate function  $\widetilde{I}$.  
\end{thmalpha}

The proof of Theorem \ref{th1}  is given in Section \ref{pf-ind-general}, with a special case treated in Section \ref{pf-ind-special}.

\begin{remark}
{\em In this paper, we discuss only sequences that satisfy Hadamard's gap condition.
If $(a_k)_{k \in \N} \subset \mathbb{N}$ is a sub-lacunary sequence, that is, increasing and 
$\log(a_k)/k \to 0$ as $k \to \infty$, then for $z \in (0,1)$, we argue below that
\begin{equation}\label{sublacunary}
  \liminf_{ n \to \infty} \frac{1}{n} \log \Pro\left(\left\{\omega\in[0,1]\,:\, \frac{1}{n} \sum_{k=1}^n \cos(2\pi a_k \omega) \geq z\right\}\right) = 0,
\end{equation}
which says that in contrast to the lacunary case, these probabilities decay slower than exponentially in $n$.
To show \eqref{sublacunary}, 
fix $z \in (0,1)$ and 
choose $\delta = \delta(z) > 0$ such that
$\cos(2\pi \omega) \geq z$ for $|\omega| \leq \delta$. 
Then, $\cos(2\pi a_k \omega) \geq z$ for all $k\in\{1,\dots,n\}$ 
if  $|\omega| \leq \delta/a_n$. 
But $\Pro\left(\omega \leq \frac{\delta}{a_n}\right) = \frac{\delta}{a_n}$ and
\begin{eqnarray*}
 && \liminf_{n \rightarrow \infty}  \frac{1}{n} \log \Pro \left(\left\{ \omega\in[0,1]\,:\, \frac{1}{n}\sum_{k=1}^n \cos(2\pi a_k \omega) \geq z \right\}\right) \cr
 & \geq & \liminf_{n \rightarrow \infty} \frac{1}{n} \log \Pro \big(\big\{\omega\in[0,1]\,:\, \cos (2 \pi a_k \omega) \geq z, \forall k \in \{1,2, \ldots, n\} \big\}\big) \\
 & \geq & \lim\limits_{ n \to \infty}\frac{1}{n} \log \Pro\left (\omega \leq \frac{\delta} {a_n}\right) \\
 & = & 0,
 \end{eqnarray*}
where the last equality uses the assumption that $(a_k)_{k \in \N}$ is sub-lacunary.  Since the opposite inequality follows trivially, this proves
 \eqref{sublacunary}.
} 
\end{remark}

\subsection{The case of geometric progressions $a_k=q^k$}
Let us now consider the case when there exists $q \in \{2,3,\ldots\}$ such that  $a_k = q^k$ for $k \in \N$. Contrary to the case of a large gap condition (see Theorem \ref{th1}), we now obtain LDPs with the same speed $n$, but with rate functions that are different from $\widetilde I$, and depend on the value of
$q$. Our main findings in this case are summarized in the following theorem, whose proof is given in Section \ref{subs-pf-theoremB}.  

\begin{thmalpha} \label{th2}
  Fix $q \in \{2,3,\ldots\}$. Let $a_k = q^k$ for $k \in \N$,  let 
  $S_n$ be the partial sum defined in \eqref{def-sn}.
  Then the following limit exists: 
\begin{equation}\label{eq:cum_gen_funct_converge}
\Lambda_q (\theta) := \lim_{n \rightarrow \infty}
\frac{1}{n} \log \E[e^{\theta S_n}],
\end{equation}
with the convergence  holding uniformly for $\theta$ in  compact subsets of an open set $\mathcal D$ in the complex plane such that $\R\subset \mathcal D$. 
 Moreover,  $(S_n/n)_{n \in\N}$ satisfies an LDP with speed $n$ and  rate function $I_q$,   
 which is the Legendre-Fenchel transform of $\Lambda_q$, that is,
$$
I_q(x) = \sup_{\theta\in\R} \left[\theta x - \Lambda_q(\theta) \right], \qquad x\in\R.  
$$
 Furthermore,  each $I_q$ satisfies $I_q(x) > 0$ for $x\neq 0$ and
$I_q$ is equal to $+\infty$ outside $[-1,1]$, and 
  the family of rate functions
    $I_q, q \in \{2, 3, \cdots\},$ has
    the following properties: 
\begin{itemize}
\item[(i)] For every $q \in \{2,3,\ldots\}$,  we have $I_q(1) \leq \widetilde{I}(1)$ and 
$I_q(x) < \widetilde{I}(x)$ for $x \in (0,1)$, 
where $\widetilde I$ is  defined in \eqref{def-widetildeI}.
In particular, the functions $I_q$ and $\widetilde I$ are different. 
\item [(ii)] The limit  $\lim_{q \to \infty} I_q(x) = \widetilde{I}(x)$ holds  uniformly on compact subsets of the interval $(-1,1)$.
\item [(iii)] There is a smooth transition of $I_q$ towards $\widetilde{I}$ as $q\to\infty$, in the following sense: 
 for all $m\in \{1,\ldots,q\}$, we have
\[
\Big(\frac{d}{d\theta}\Big)^m \Lambda_q(\theta)\Big|_{\theta=0} = \Big(\frac{d}{d\theta}\Big)^m \widetilde{\Lambda}(\theta)\Big|_{\theta=0} = \kappa_m(\widetilde{X}_1),
\]
where $\widetilde \Lambda$ is defined in \eqref{def-tildeLambda}, $\widetilde{X}_1$ in \eqref{def-wtildexk}, $\kappa_m(\widetilde{X}_1)$ is the $m^{\text{th}}$ cumulant of $\widetilde{X}_1$. Further,
$$
\Big(\frac{d}{d x}\Big)^m I_q(x)\Big|_{x=0} = \Big(\frac{d}{d x}\Big)^m \widetilde{I}(x)\Big|_{x=0}.
$$
In other words, the coefficients of $1,\theta,\ldots,\theta^q$ in the Taylor expansions of $I_q(\theta)$ and $\widetilde{I}(\theta)$  coincide at the origin. 

\item[(iv)] Whereas in (iii) the first $q$ derivatives of $\Lambda_q$ and $\widetilde{\Lambda}$ coincide, this is no longer true for the derivative of order $q+1$. In particular, 
$$
\Big(\frac{d}{d\theta}\Big)^{q+1} \Lambda_q(\theta)\Big|_{\theta=0}
=
\Big(\frac{d}{d\theta}\Big)^{q+1} \widetilde{\Lambda}(\theta)\Big|_{\theta=0}+ \frac{q+1}{2^{q}}
>
\Big(\frac{d}{d\theta}\Big)^{q+1} \widetilde{\Lambda}(\theta)\Big|_{\theta=0}.
$$
\end{itemize}
\end{thmalpha}

\vskip 2mm
\begin{figure}
  \label{fig}
\begin{center}
\includegraphics[height=6.91cm]{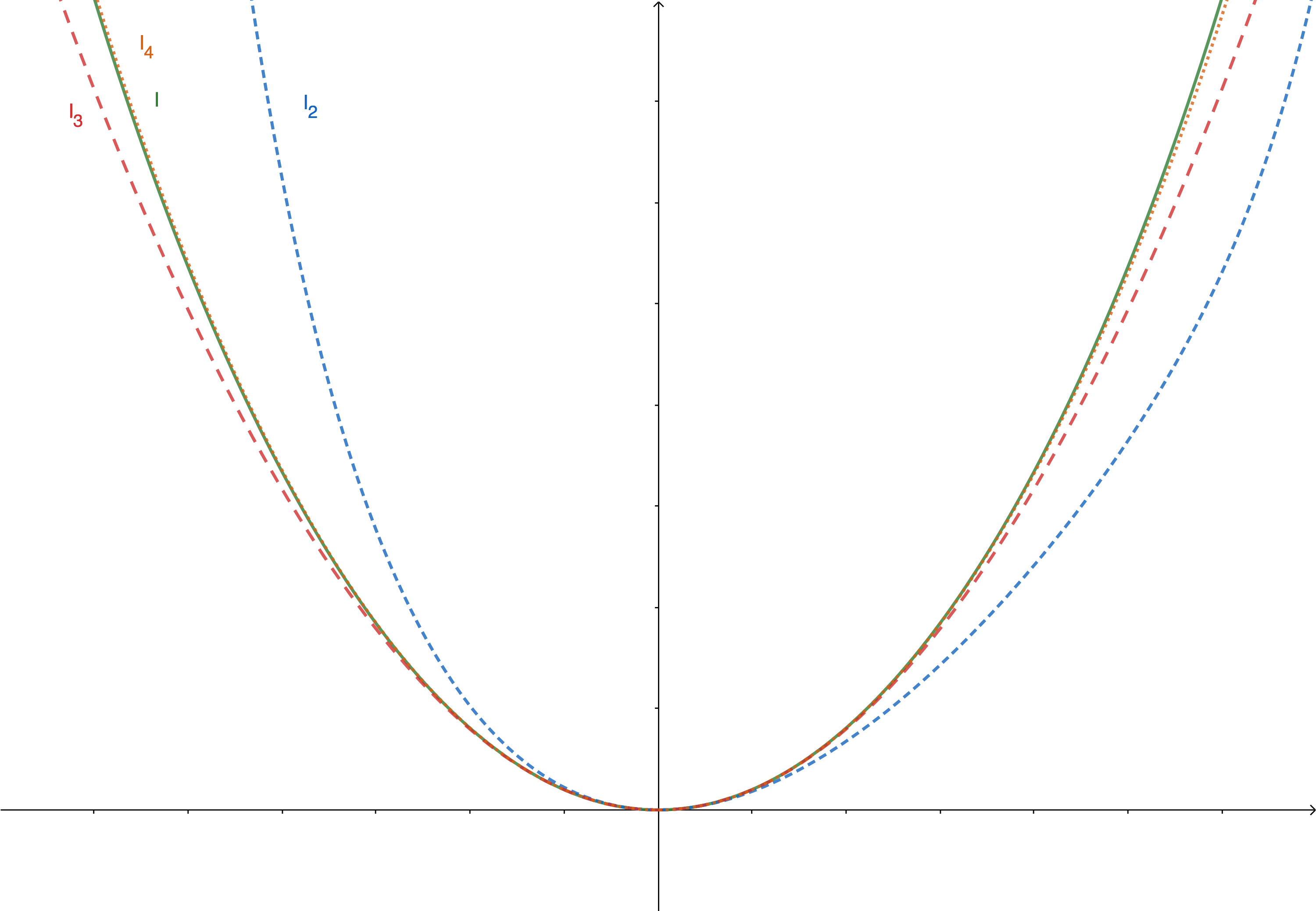} \\
   
   \begin{flushleft}
   Figure~\ref{fig}. Approximations to the rate function $\widetilde{I}$ (green) corresponding to the case of independent random variables and to the first three rate functions $I_2$ (blue), $I_3$ (red), and $I_4$ (orange) that appear in Theorem \ref{th2}.
   \end{flushleft}
\end{center}
\end{figure}

Indeed, we will see in Proposition \ref{prop:I_2_I_3} that not only is $I_q \neq \widetilde I$ but
it is also true that $I_{q_1} \neq I_{q_2}$ if $q_1\neq q_2$.

We comment on  Theorem \ref{th2}. The fact that $(S_n/n)_{n \in\N}$ satisfies an LDP will be deduced from general results
on thermodynamic formalism and expanding maps of the interval $[0,1]$ (see
the proof of the theorem in Section \ref{subs-pf-theoremB}).
 The key takeaways 
of the theorem are the properties of the rate functions in (i)--(iv).
Parts~(ii) and (iii) state that  the rate functions $I_q$ converge towards $\widetilde{I}$ as $q\to\infty$, which is in accordance with the ``limiting case'' of Theorem \ref{th1} where the ratio of $a_{k+1}/a_k$ diverges to $+\infty$, and where the rate function for the lacunary sums coincides with the one for the truly independent case. The rate function $\tilde{I}$ and $I_q$ for small
  $q$,  are illustrated in Figure~\ref{fig}.

Note also that as a consequence of conclusion (i) of Theorem \ref{th2}, the probability of large deviations of the lacunary sums $S_n$ is (asymptotically) greater than the  large deviation probability for the corresponding partial sum
of independent  random variables $\widetilde{S}_n$ defined in  \eqref{def-wtildesn}. 
However, since the statement (i) only applies to positive values of $x$, this conclusion is only valid for large \emph{positive} deviations of the lacunary sum. In the case of large \emph{negative} deviations there seems to be an interesting dichotomy. When $q$ is \emph{odd}, then the lacunary sums have a distribution symmetric around $0$, which is a consequence of the fact that the mapping $\omega \mapsto \omega +1/2 \pmod 1$ transforms the function  $\cos (2 \pi q^k \omega)$ into  $\cos (2 \pi q^k (\omega  +1/2)) = - \cos (2 \pi q^k \omega)$. 
Accordingly, the probabilities of large positive and large negative deviations are equal, and we have $I_q(-x) = I_q(x)$, so that in the odd case we have $I_q (x) \leq \widetilde{I}(x)$ for all $x \neq 0$, with $I_q(x) < \widetilde{I}(x)$ for all sufficiently
small $|x|$ due to  (iv). In contrast, when $q$ is \emph{even}, there is no such symmetry. In fact, for even $q$, it follows from (iii) and (iv) that $\Lambda_q(\theta) < \widetilde \Lambda(\theta)$ for $\theta<0$ sufficiently close to $0$ (because the coefficient of $\theta^{q+1}$, an odd power, in the Taylor series of $\Lambda_q(\theta)$ is larger than that of $\widetilde \Lambda(\theta)$, while the smaller powers coincide). By taking the corresponding Legendre-Fenchel transforms, it follows that $I_q(x)>\widetilde I(x)$ for $x<0$ sufficiently close to $0$.
We believe the above inequalities hold without restricting $|x|$ to be sufficiently small, as stated below in this conjecture:

\begin{conj}
Let $q \in \{2,3,\ldots\}$ and let $a_k = q^k$ for $k \in \N$. Let $I_q$ be the rate function in the LDP for $(S_n/n)_{n \in\N}$ (which exists by Theorem \ref{th2}).
Then, if $q$ is \emph{odd},
$$
I_q (x) < \widetilde{I}(x) \quad \text{for all} \quad x\in (-1,1)\backslash \{0\}.
$$
On the other hand, if $q$ is \emph{even}, then
$$
I_q (x) < \widetilde{I}(x) \quad \text{for all} \quad 0<x<1
\qquad
\text{ and }
\qquad
I_q (x) > \widetilde{I}(x) \quad \text{for all} \quad -1<x<0.
$$
\end{conj}

As  argued above, we have $\widetilde I(\pm 1) = +\infty$. Since $|S_n/n|\leq 1$, it is clear that $I_q(x) = +\infty$ for $|x|>1$.  The next lemma,
whose proof is given at the end of Section \ref{subs-pf-theoremB}, 
states that $I_q(+1)$ is finite.
\begin{lemma}\label{lem:I_q_+1}
For all $q \in \{2,3,\ldots\}$  we have $I_q(+1) \leq \log q$.
\end{lemma}

As explained above, for odd $q$ we have $I_q(-1)= I_q(+1)$. For even $q$, it remains unclear whether $I_q(-1)$ is finite (and in fact, it is not even clear whether $I_q(x)$ is finite for all $-1<x<0$).

\vspace*{2mm}
The functions $\Lambda_q$ and $I_q$ appearing in Theorem~\ref{th2} are not really explicit. In fact, the only known  formula for $\Lambda_q$ seems to be its representation as the logarithm of the largest eigenvalue of a certain Perron-Frobenius operator (see the proof of Theorem~\ref{th2} in Section \ref{subs-pf-theoremB}).  The next proposition identifies the first few terms in the Taylor expansions of $I_q$ for $q\in \{2,3,4\}$.
Before stating it, let us look at the Taylor series of the rate function $\widetilde I$. From the expression for $\widetilde I$  in \eqref{def-widetildeI} and  properties of the Legendre-Fenchel transform, it follows that the derivative $\widetilde I'$ of $\widetilde I$ is the inverse function of the derivative $\widetilde \Lambda'$ of $\widetilde \Lambda$, see \cite[Corollary 23.5.1, p. 219]{R1970}.     
 Using this fact together with the expression for $\widetilde{\Lambda}$ in \eqref{arcsin_cumgf}, which yields the series expansion 
\begin{equation}
\widetilde \Lambda (\theta)
=
\frac{\theta^2}{4}-\frac{\theta^4}{64}+\frac{\theta^6}{576}-\frac{11 \theta^8}{49152}+\frac{19 \theta^{10}}{614400}-\frac{473 \theta^{12}}{106168320}+O\left(\theta^{14}\right),
\qquad
\mbox{ as } \theta\to 0,
\end{equation}
 one can easily compute the first few terms in the Taylor series of $\widetilde I$ near $0$:
\begin{equation}\label{eq:I_tilde_taylor}
\widetilde I(z) = z^2+\frac{z^4}{4}+\frac{5 z^6}{36}+\frac{19 z^8}{192}+\frac{143 z^{10}}{1800}+\frac{1769 z^{12}}{25920}+
O(z^{14}),
\qquad
\mbox{ as } z\to 0.
\end{equation}

\begin{prop}\label{prop:I_2_I_3}
In the case when $a_k = 2^k$ for all $k\in\N$, the Taylor expansion of the rate function around $0$ is given by
$$
I_2(z) = z^2-z^3+\frac{3 z^4}2-\frac{13 z^5}6+\frac{29 z^6}9-\frac{23 z^7} 5+\frac{1127 z^8}{180} - \frac{29083 z^9}{3780}+\frac{12077 z^{10}}{1575} + O(z^{11}).
$$
In the case when $a_k = 3^k$ for all $k\in\N$, the rate function satisfies
$$
I_3(z) = z^2- \frac{z^4}{12} + \frac{z^6}{6}- \frac{39 z^8}{320} + \frac{18113 z^{10}}{100800} + O(z^{12}), 
$$
whereas  when $a_k = 4^k$ for all $k\in\N$, we have
$$
I_4(z) = z^2+ \frac{z^4}{4} - \frac{z^5}{12} + \frac{5 z^6}{36} + O(z^7).
$$
In particular, the functions $I_2,I_3,I_4$, and $\widetilde I$  all differ from each other. 
\end{prop}

In fact, in the proof of this proposition, which is deferred to 
Appendix \ref{ap-a}, 
we  describe an algorithm  to compute an arbitrary number of terms in the Taylor expansion of $I_q$ 
 for every $q\in \{2,3,\ldots\}$. The algorithm, as well as the proof of properties (i)--(iv) in 
Theorem~\ref{th2},  are based on an analysis of the number of representations of $0$ as a sum of $m$ terms of the form $\pm q^1, \pm q^2,\ldots, \pm q^n$. Denoting this number by $A_m(n)$, we  prove in Proposition~\ref{prop:a_m_n} that for fixed $m\in\N$, it is a polynomial in $n$ for all $n\geq m-2$.  This fact allows us to compute the first few moments of $S_n$ and prove the above expansions.

\vskip 2mm
Let us recall from~\eqref{thm:clt salem zygmund} that $(S_n/\sqrt n)_{n\in \N}$ satisfies a central limit theorem  under the Hadamard gap condition \eqref{Hadamard-gap}. The next theorem states that, perhaps
surprisingly, the LDP does not hold in the same generality. More precisely, by mixing up powers of $2$ and $3$ we shall obtain an example of an Hadamard gap sequence $(a_k)_{k\in\N}$ for which the corresponding scaled partial sums $(S_n/n)_{n\in\N}$ fail to satisfy an LDP.
This is stated in the following result,
which is proved in Section \ref{subs-pf-theoremC}.

\begin{thmalpha} \label{th3}
  There exists a sequence of positive integers $(a_k)_{k \in\N}$ satisfying $a_{k+1} / a_k \geq q$ 
for some $q>1$ and all $k\in\N$, for which $(S_n/n)_{n \in\N}$ does not satisfy an LDP with speed $n$. More precisely, for this sequence $(a_k)_{k\in\N}$ there exists $\bar{x}_0\in (0,1)$ such that for all $x_0 \in (0,\bar{x}_0)$,  
$$
0< \liminf_{n \to \infty} -\frac 1n \log \mathbb{P} (S_n > n x_0) < \limsup_{n \to \infty} -\frac 1n \log \mathbb{P} (S_n > n x_0) <\infty.
$$
\end{thmalpha}
  Note that  if $(S_n)_{n \in \N}$ did satisfy  an LDP with speed  $n$
  and rate function $I$,  then the last display would imply that 
   $\inf_{x \in [0,x_0)} I(x) \neq \inf_{x \in [0,x_0]} I(x)$ for all
    $x_0  \in (0, \bar{x}_0)$, which, in turn, implies that
    $I$ is not continuous at any point in $(0,x_0)$. This leads to
    a contradiction since 
     $I$ must be lower-semicontinuous, thus showing that
    $(S_n)_{n \in \N}$ does not satisfy  an LDP with speed  $n$. 

\subsection{Randomized perturbation}
\label{subs-randperb}

Theorem \ref{th1} and Theorem \ref{th2} might together give the impression that the existence of a limit of $a_{k+1}/a_k$ as $k\to\infty$ ensures an LDP
  and its value determines the rate function. In particular, one might be tempted to conjecture that the condition $a_{k+1}/a_k \to \infty$ is  necessary for the rate function 
in the LDP to coincide with $\widetilde{I}$, the corresponding rate function for the independent case.
However, this is not true.
  As Theorem \ref{th4} below shows, it is possible to construct lower order random perturbations  
$(a_k)_{k \in\N}$ of the sequence $(2^k)_{k \in \N}$  with $\lim_{k \rightarrow \infty} a_{k+1}/a_k = 2$,  
for which the corresponding
sequence $(S_n/n)_{n \in \N}$ almost surely satisfies an LDP with speed $n$ and  rate function $\widetilde{I}$. 
This shows that a random perturbation may completely destroy the underlying dependence, at least at the large deviation
scale, and, as also further elaborated in Section \ref{sec-conclusion},
  rather than the asymptotic growth rate of the lacunary sequence $(a_k)_{k\in \N}$, what seems to
  determine the form of the rate function (when an LDP holds) is the arithmetic structure of   $(a_k)_{k\in \N}$.

\begin{thmalpha} \label{th4} 
  Suppose we are given a  sequence $Y = (Y_k)_{k \in \N}$  of independent~random variables, 
  with each $Y_k$  uniformly distributed on the discrete set
\begin{equation}
  \label{def-discset}
\discset_k := \left\{h 2^{\lceil k^{2/3} \rceil}: h \in \mathbb{Z}, 0 \leq h \leq 2^{\lceil k^{2/3} \rceil} \right\}, \quad   k \in \N, 
\end{equation} 
all supported on a common  probability space $(\yspace, \yfilt, \py)$,
and an independent   random variable $U \sim \mathrm{Unif}(0,1)$.  
Also,  for each $y \in \discsetinf :=$   $\otimes_{k\in\N}\discset_k = $  $\{(y_k)_{k\in\N}\,:\,y_k\in \discset_k\}$,  
define  $a^y_k := 2^k + y_k$ for all $k \in\N$, and
let
\[  S^y_n := \sum_{k=1}^n \cos (2 \pi a_k^y U), \quad n \in \N. \]
Then, for $\py \circ Y^{-1}$-almost every $y \in \discsetinf$,  
the sequence $(S^y_n/n)_{n \in\N}$,  satisfies an LDP with speed $n$ and rate function $\widetilde{I}$.
\end{thmalpha}

In large deviation parlance, the LDP in Theorem \ref{th4} is often referred to as a ``quenched LDP'' since the LDP is 
  conditional on the realization of the sequence $y$, and  not averaged over the randomness of $Y$. 
  Note however, that although  the sequence $(S^y_n/n)_{n \in\N}$ depends on the choice of $y = (y_k)_{k \in \N}$, the rate function  $\widetilde{I}$  of the LDP (which holds for $\py \circ Y^{-1}$-almost every $y$)  does not. 
  The proof of  Theorem \ref{th4} is given in  Section \ref{subs-pf-TheoremD}.
  
\begin{remark}
\label{rem-thD}
{\em 
Note that for every realization $y$ of $Y$  in Theorem \ref{th4}, 
  we have  $2/(1+2^{-k^{1/3}}) \leq a^y_{k+1}/a^y_k \leq 2(1+2^{-(k+1)^{1/3}})$, which implies that, as $k \to \infty$,
  $a_{k+1}^y/a_k^y \to 2$. 
 Thus, Theorem \ref{th4} proves that there exist lacunary sequences
with proper exponential growth (as opposed to super-exponential growth as in Theorem \ref{th1}) that satisfy the LDP with rate function $\widetilde{I}$.
}
\end{remark}
Also, by interleaving the sequence $(2^k)_{k\in\N}$ with the sequence $(a_k^y)_{k\in\N}$ constructed in Theorem~\ref{th4} (in the same way as in the proof of Theorem~\ref{th3}), and using the fact that  $I_2$ does not coincide with $\widetilde I$, it is possible to construct a sequence $(b_k)_{k\in\N}$ such that as $k \to \infty$, $b_{k+1}/b_k \to 2$, but the corresponding lacunary sums do not satisfy an LDP. This sharpens Theorem~\ref{th3}.
Somewhat surprisingly,  even the randomized construction in Theorem \ref{th4}  seems to be quite sensitive.  
While Theorem \ref{th4} can certainly be generalized in many directions, it appears to be more challenging to  prove an analogue when each random component $Y_k$ is sampled from $\{0, 1, 2, \dots, k\}$, or when
it is  sampled from $\{0, 2^k, 2\cdot 2^k, \dots, k \cdot 2^k\}$. As elaborated in the next section, this can be related to the number of solutions of certain corresponding Diophantine  equations
   (see also the proof of Theorem \ref{th4} in Section \ref{subs-pf-TheoremD}).

\subsection{Concluding remarks and further open questions}
\label{sec-conclusion}

\subsubsection{Connection between LDPs and Diophantine equations}

Our results (in particular, Theorems B, C, and D together)  
show that only knowing that $\lim_{k \to \infty} a_{k+1}/a_k = \eta$ for some $\eta>1$ does in general not allow one to
determine the rate function in the LDP for the lacunary sum, or even conclude the existence of an LDP.
The proofs of these results, which are presented in Section \ref{sec-pfs}, 
  often involve  approximating  the exponential function in $\E[e^{\theta S_n}]$ via   a Taylor polynomial in $\theta$. 
  In turn, keeping in mind from \eqref{def-xk} and \eqref{def-sn} that $S_n = \sum_{k=1}^n X_k$
  is a finite sum of trigonometric functions,
  this entails estimates of integrals of products of trigonometric polynomials. Due to the orthogonality of the trigonometric system, calculation of these integrals leads to counting the number of solutions to certain Diophantine equations (with $a_k$'s as variables). The reason why the rate function in the LDP (when it exists) for some lacunary sequences
differs from the one for the independent case may be attributed to the existence of too many solutions to these Diophantine equations. For example, when $a_k = 2^k$ for all $k$, then the equation $2 a_{k} - a_{\ell} = 0$ holds for many combinations of $\ell, k$, namely $\ell = k+1$ for all $k$. The Diophantine equations that appear in this context are always linear homogeneous Diophantine equations with integer coefficients. Thus, there are
many more solutions to such equations when the sequence $(a_k)_{k \in \N}$ allows many quotients $a_\ell / a_k$ that are integers. 
In contrast, when the quotients $a_\ell / a_k$ are bounded away from any integer (and any rational with a small denominator), then
these Diophantine equations would have fewer solutions. Thus, while specific random perturbations such as the one chosen
  in Theorem \ref{th4} may drastically diminish the number of solutions, any generalization of Theorem \ref{th4} would require
  determining  precisely how the Diophantine structure
  is altered by an arbitrary random perturbation, which appears to be highly non-trivial. 
Further, this also suggests that there may still be some information that can be gleaned from the existence of the
limit $\lim_{k \to \infty} a_{k+1}/a_k = \eta$, but only when $\eta$ is a number that is not well approximated by rationals with small denominators, and when the same is true for $\eta^2,\eta^3, \dots$, which correspond to the limits of $a_{k+2}/a_k$, $a_{k+3}/a_k$, and so on.
We formulate this as an open problem.
\begin{problem} \label{th5}
Let $(a_k)_{k \in\N}$ be a lacunary sequence and assume that $a_{k+1}/a_k \to \eta$ for a transcendental number $\eta>1$ (i.e., $\eta$ is not the root of a non-zero polynomial with integer coefficients). Is it true that $(S_n/n)_{n\in\N}$ satisfies an LDP with speed $n$ and rate function $\widetilde{I}$ (i.e., with the same rate function as in the independent case, that is, as for $(\widetilde{S}_n)_{n \in \N}$)?
\end{problem}
It may be that stronger assumptions on $\eta$, such  as a condition on the irrationality measure of $\eta$ and its powers, 
  are necessary to derive the desired conclusion.
  However, we think that such an additional assumption should not be required.
  On the other hand, we believe that just assuming $a_{k+1}/a_k \to \eta$ for an algebraic irrational $\eta>1$ will not be sufficient to deduce an LDP with rate function $\widetilde{I}$. 

  Since the Diophantine structure of the sequence $(a_k)_{k \in\N}$ plays such a key role in establishing the LDP for lacunary trigonometric sums, it would be very interesting to study this phenomenon in more detail. A natural candidate to analyze is the sequence $a_k = 2^k +1,~k \geq 1,$ known from the Erd\H os-Fortet example mentioned in the Introduction. In the context of the CLT and LIL, this sequence and its generalizations have received widespread attention. The different type of behavior resulting from ``pure'' geometric progressions such as $a_k=2^k,~k \geq 1,$ on the one hand, and ``perturbed'' sequences such as $a_k = 2^k + 1,~k \geq 1,$ on the other hand, can be explained analytically in terms of Fourier analysis. However, there is also a very interesting dynamical perspective, where the pure geometric progressions allow a natural interpretation as an ergodic sum, while the perturbed sequences have been interpreted as modified ergodic sums; see for example, \cite{conz, fuku_m, schnell}.\\

\subsubsection{Normal number theory}

From a number theoretic perspective, sequences of the form $(q^k \omega)_{k \in\N}$ are associated with the notion of normal numbers (in base $q$), as introduced by Borel in 1909. It is well known that almost all numbers are normal in any base. The degree of normality of a number can be quantified using uniform distribution theory and discrepancy theory, which by Weyl's criterion and the Erd\H os-Tur\'an inequality naturally leads to trigonometric sums such as the ones studied in the present paper (see \cite{dts,kn} for general background on uniform distribution modulo one and discrepancy theory). LDPs for such sums can thus be viewed as quantifying the relative fraction of ``non-normal'' or ``abnormal'' numbers in a certain base, that is, numbers whose digital structure very significantly deviates from ``normal'' behavior. Such non-normal numbers have been intensively studied in the number theory literature, see for example \cite{agi,mart,ols}. A particularly challenging and interesting topic in normal number theory are questions concerning simultaneous normality resp.\ non-normality in two or more different bases (see for example \cite{ck,poll}). In terms of the large deviation problems studied in the present paper, it would be interesting to quantify the proportion of numbers that are non-normal in two or more different bases.
For example, one could try to establish an LDP to estimate
the probability of the set where two lacunary sums arising from the sequences $(q^k)_{k \in \N}$ and $(r^k)_{k \in \N}$ (for two different bases $q,r \geq 2$) are both large.

In the context of normal numbers, the case of general sequences $(a_{k})_{k \in\N}$ satisfying $\frac{a_{k+1}}{a_k} \in \mathbb{Z}_{\geq 2},~k \geq 1,$ corresponds to normality with respect to so-called Cantor expansions. This is a topic that
has been pioneered by Erd\H os and R\'enyi \cite{erdos-renyi,renyi}, and  received strong attention in recent years; see for example \cite{cantor2,cantor3,cantor1} for recent work, and cf.\ also our proof of a special case of Theorem A in Section \ref{pf-ind-special} below.

\subsubsection{More general lacunary sums}

We finally recall from the introduction that the theory of lacunary trigonometric sums is structurally relatively simple in comparison with the theory of general lacunary sums, where interesting new phenomena show up even in the CLT setting.  In light of this, it would be interesting to study the LDP for 
$$
S_n(\omega) = \sum_{k=1}^n f(a_k \omega),
$$
where $f$ is a centered 1-periodic function (possibly satisfying some regularity assumptions). Already when $f$ is a 2-term trigonometric polynomial (as in the Erd\H os-Fortet example alluded to above) there can be additional arithmetic effects in comparison to the simple case of pure trigonometric sums. It would certainly be interesting to investigate LDPs in this general lacunary setup.  A further challenging step would be to go beyond lacunary sums for a single fixed function $f$ and investigate LDPs for the discrepancy (which is defined as a supremum over indicator functions), in the spirit of Philipp's \cite{phil} resolution of the Erd\H os--G\'al conjecture and Fukuyama's \cite{fuku} very precise results for the LIL for geometric progressions $a_k=q^k$. \\

\section{Proofs}\label{sec-pfs}
In our proofs we will make use of the G\"artner-Ellis theorem, which we require in the following form. For a reference, see, for example, \cite[Theorem 2.3.6]{dembo_zeit}.
\begin{thm}[G\"artner-Ellis theorem]
\label{th-GE}
Let $(S_n)_{n \in\N}$ be a sequence of real-valued random variables. Suppose that the limit
$$
\Lambda(\theta) = \lim_{n \to \infty} \frac{1}{n} \log \E \left[ e^{\theta S_n} \right]
$$
exists  for all $\theta \in \R$. 
Assume furthermore that the function $\theta \mapsto \Lambda(\theta)$ is differentiable for all $\theta \in \R$. Then $(S_n/n)_{n \in\N }$ satisfies an LDP  with speed $n$ and  convex  rate function $I$,  which can be expressed as the Legendre-Fenchel transform of $\Lambda$, that is,
$$
I(x) = \sup_{\theta \in \R} \left[\theta x - \Lambda(\theta)\right] \in \R\cup\{+\infty\},
\qquad x\in\R.
$$
\end{thm}

\subsection{Proof of Theorem~\ref{th1} in a simple special case}
\label{pf-ind-special}
We first give a proof of Theorem~\ref{th1} in a special case, the justification being two-fold: we believe that the proof helps the intuition of the reader, but we also point out that it goes through if we replace $\cos(2\pi \cdot)$ by any Lipschitz continuous function $f$ that  also satisfies  \eqref{Kac-conditions} (i.e., is $1$-periodic and centered).
We consider a sequence $(a_k)_{k\in\N}$ of positive integers such that
$a_1 = 1$, 
\[
m_k :=  \frac{a_{k+1}}{a_k}\in \{2,3,\ldots\}, k \in \N,  \qquad\text{and}\qquad \lim_{k\to\infty} \frac{a_{k+1}}{a_k}=+\infty. 
\]
The assumption $a_1 = 1$ is without loss of generality, but the  assumption  
that  $m_k, k \in \N,$ are integers will be essential for the following argument. 
By the G\"artner-Ellis theorem, it suffices to show that for all $\theta\in\R$, 
$$
\lim_{n \to \infty} \frac{1}{n} \log \E \left[ e^{\theta S_n} \right] = \widetilde \Lambda(\theta), 
$$
with $\widetilde{\Lambda}$ as defined in \eqref{def-tildeLambda}.  
To this end, we shall approximate each $S_n$ by a  random variable $T_n$
that is easier to deal with, in the sense that it 
can be written as  a sum of 
independent random variables expressed, as defined below, in terms of
certain conditional expectations.
First, recall that the Cantor series expansion $(\xi_k)_{k \in \N}$ of
  $\omega \in [0,1)$ associated with $(m_{k})_{k \in \N} \subset \{2,3, \ldots\}$ is given as follows:
  \[   \omega = \sum_{k=1}^\infty \frac{\xi_k(\omega)}{m_1 m_2 \cdots m_{k}}
  = \sum_{k=1}^\infty \frac{\xi_k(\omega)}{a_{k+1}}, 
  \]
  with $\xi_k(\omega)  \in \{0,1,\ldots, m_{k}-1\}$ for every $k \in \N$.
  This expansion was first introduced by Cantor in \cite{Can1869}, and
  the investigation of  its probabilistic properties
  appears to have been initiated by Erd\H{o}s,  R\'{e}nyi, and Turan \cite{erdos-renyi,renyi,Turan56}. 
  Our construction uses the following 
  key property established by R\'{e}nyi in \cite{renyi} (see also (7) of
  \cite{ErdosRenyi59}): 
  the image of Lebesgue measure on $[0,1)$, 
     under  the correspondence  $\omega \mapsto (\xi_k(\omega))_{k \in \N}$
    makes $(\xi_k)_{k \in \N}$ a sequence of independent integers with each 
    $\xi_k$ uniformly distributed on $\{0,1, \ldots, m_{k} - 1\}$.\footnote{This is simply a generalization of the possibly more familiar result going back to Borel  \cite{Borel1909}, where $m_k = r$ for all $k$ and the correspondence 
    between elements of $[0,1]$ and their $r$-ary expansions 
    maps Lebesgue measure on $[0,1]$ to the Bernoulli product measure
    on the space of $\{0,1, \ldots, r-1\}$-valued sequences, with uniform marginals (see also \cite[Section 2.3]{PL2020} for
     a  more detailed exposition).} 

Given this property,  now consider the filtration $\mathcal F_{2} \subset \mathcal F_{3} \subset\ldots$, where the $\sigma$-algebras are defined by
\[
\mathcal F_{k+1} := \sigma\bigg( J_{k+1,i}, i=0,\dots,a_{k+1}-1 \bigg),  \quad k \in \N, \]
where for $k \in \N$,
\[  J_{k+1,i}:= \left[ \frac{i}{a_{k+1}}, \frac{i+1}{a_{k+1}}\right), \quad  
i=0,\dots,a_{k+1}-1.
\]
We now define certain conditional expectations: 
\[
Y_k := \E\big[ X_k| \mathcal F_{k+1}\big] = \sum_{i=0}^{a_{k+1}-1}\E\Big[X_k\Big| J_{k+1,i}\Big]\mathbbm 1_{J_{k+1,i}}\,,  \quad k \in \N. 
\]
In particular, we see that by construction, for every $k\in\N$ and  $i=0,\dots,a_{k+1}-1$,
\[   Y_k = \sum_{i=0}^{a_{k+1}-1} c_{k+1,i}  \mathbbm 1_{J_{k+1,i}}\,,   \]
where  $c_{k+1,i}$, the constant representing the value of $Y_k$ on $J_{k+1,i}$, is given by 
\[
c_{k+1,i} :=  a_{k+1}  \int_{J_{k+1,i}} X_k(\omega) \, \lambda(d\omega). 
\]
Since the function $X_k$ is $(1/a_k)$-periodic, for any $k \in \N$,  $c_{k+1,i} = c_{k+1,i'}$ whenever $|i-i'|$ is a multiple of $m_k = a_{k+1}/a_k$. 
Hence, for each $k \in \N$, the random variable $Y_{k}$ is only a function of $\xi_k$.  Since the
$\{\xi_k\}_{k \in \N}$ are independent, 
  the random variables $\{Y_k\}_{k \in \N_0}$ are also independent.

We now show  that the approximation of $X_k$ by $Y_k$ is sufficiently good, more precisely, for each $k\in\N$ and  $i=0,\dots,a_{k+1}-1$,
using  the mean-value theorem and the fact that $X_k$ has Lipschitz constant $2\pi a_k$,  we have 
\begin{align*}
\max_{\omega \in J_{k+1,i}} |X_k(\omega) - Y_k(\omega)|
&= \max_{x\in J_{k+1,i}} \bigg|X_k(\omega) - a_{k+1}\int_{J_{k+1,i}} X_k(y)\,dy\bigg| \\
& = \max_{x\in J_{k+1,i}}\big|X_k(\omega) - X_k(\omega_0)\big| \\
&  \leq \max_{x\in J_{k+1,i}} 2\pi a_k \big|\omega-\omega_0\big| \\
&  \leq 2\pi \frac{a_k}{a_{k+1}},
\end{align*}
where $\omega_0 = \omega_{0,k,i} \in J_{k+1,i}$ is obtained from the mean value theorem.  Taking the maximum over all $i=0,\dots,a_{k+1}-1$ yields
\begin{equation}\label{eq:X_k_Y_k}
\| X_k -Y_k\|_\infty \leq  2\pi\, \frac{a_k}{a_{k+1}}.
\end{equation}
In particular, this means that if $S_n:= \sum_{k=1}^nX_k$ and $T_n:=\sum_{k=1}^n Y_k$, then
\[
\|S_n-T_n\|_\infty \leq 2\pi \sum_{k=1}^n \frac{a_k}{a_{k+1}} = o(n), \qquad n\to\infty,
\]
because by assumption $a_k/a_{k+1}\to 0$ as $k\to\infty$. 
For fixed $\theta\in\R$ we obtain
$$
\E \left[e^{\theta S_n}\right] = \E \left[e^{\theta T_n} e^{\theta (S_n-T_n)}\right] \leq e^{|\theta| \|S_n-T_n\|_\infty} \E \left[e^{\theta T_n} \right].
$$
We  also have the analogous lower bound  
$$
\E \left[e^{\theta S_n}\right]= \E \left[e^{\theta T_n} e^{\theta (S_n-T_n)}\right] \geq e^{-|\theta| \|S_n-T_n\|_\infty} \E \left[e^{\theta T_n} \right].
$$
Altogether, taking into account that $\|S_n-T_n\|_\infty = o(n)$, we obtain  
\begin{equation}\label{eq:S_n_T_n}
\E \left[e^{\theta S_n}\right] = e^{o(n)}  \E \left[e^{\theta T_n} \right],
\qquad n\to\infty.
\end{equation}
Since $T_n = Y_1+\ldots+Y_n$ is a sum of independent random variables, it follows that  
\begin{equation}\label{eq:frac1n_cum_gf_S_n}
\frac 1n \log \E \left[e^{\theta S_n}\right]
=
o(1) + \frac 1n \log \E \left[e^{\theta T_n} \right]
=
o(1) +\frac 1n \sum_{k=1}^n \log \E \left[e^{\theta Y_k} \right].
\end{equation}
Similarly, in view of~\eqref{eq:X_k_Y_k} and the fact that
 $(X_k)_{k \in \N}$ are identically distributed, we have 
$$
\log \E \left[e^{\theta Y_k} \right] =  o(1) + \log \E \left[e^{\theta X_k} \right] = o(1) + \log \E \left[e^{\theta X_1} \right],
\qquad k\to\infty.
$$
Inserting this into~\eqref{eq:frac1n_cum_gf_S_n} and recalling that the usual convergence implies convergence of arithmetic means to the same limit,
the fact that  $X_1$ and $\widetilde{X}_1$ are identically distributed and the definition \eqref{def-tildeLambda} of $\widetilde{\Lambda}$, we arrive at
$$
\frac 1n \log \E \left[e^{\theta S_n}\right]
=
o(1) + \log \E \left[e^{\theta X_1} \right]
=
o(1) + \log \E \left[e^{\theta \widetilde{X}_1} \right]
=
o(1) + \widetilde \Lambda(\theta),
$$
as desired.  Since the function $\widetilde{\Lambda}$ is differentiable by \eqref{arcsin_cumgf}, the G\"artner-Ellis theorem (reproduced
as Theorem \ref{th-GE} herein) can be applied and the proof of Theorem~\ref{th1} in the case when $a_{k+1}/a_k$ are positive integer numbers tending to infinity is complete.

\subsection{Proof of Theorem~\ref{th1} in full generality}
\label{pf-ind-general}

Fix $\theta \in \R$ and a sufficiently small $\ve\in (0,1)$. As in the simple case, we wish to apply the G\"artner-Ellis theorem, but this time the analysis is more delicate.
In contrast to the proof for the simple case, which relied on a reduction to the independent setting, our proof for the general case uses harmonic analysis methods and is more in the spirit of the classical works of Salem and Zygmund, Kac, and others mentioned in the Introduction.
Recall from \eqref{def-xk} and \eqref{def-sn} that
\begin{equation}
  \label{mgf-sn}
\E \left[ e^{\theta S_n} \right] = \int_0^1 e^{\theta \sum_{k=1}^n \cos (2 \pi a_k \omega)} d\omega = \int_0^1 \prod_{k=1}^n e^{\theta \cos (2 \pi a_k \omega)} d\omega. 
\end{equation}

We start with an elementary lemma on 
approximation of the exponential function by a Taylor polynomial of length $d$.
For $d \in \N$, define 
\begin{equation}
  \label{def-poly}
  p_d(x) := \sum_{m=0}^d \frac{x^m}{m!}, \qquad x\in\R.
\end{equation}

\begin{lemma}
  \label{lem-taylor}
  There exists $d:(0,1) \to \N$ with $d(\ve) \rightarrow \infty$ as
  $\ve \rightarrow 0$ such that
  the polynomial 
  $p := p_{d(\ve)}$ satisfies for every $k \in \N$, 
  \begin{equation} \label{eps_approx}
    1 - \ve \leq \frac{p(\theta \cos (2 \pi a_k \omega))}{e^{\theta \cos (2 \pi a_k \omega)}} \leq 1 + \ve, \qquad \omega \in [0,1].  
  \end{equation}
 \end{lemma}
\begin{proof}
  Fix $d \in \N$ and $p = p_d$.
  Then, by the classical theory of Taylor approximation, the remainder in the Lagrange form satisfies
$$
\left| e^x - p(x)\right| \leq \frac{e^{\max\{0,x\}}}{(d+1)!} |x|^{d+1},
$$
and hence, 
$$
1 - \frac{e^{\max\{0,x\}} |x|^{d+1}}{(d+1)! \,e^x} \leq \frac{p(x)}{e^x} \leq 1 + \frac{e^{\max\{0,x\}} |x|^{d+1}}{(d+1)! \,e^x}\,.
$$
Noting that in our situation we have $|\theta \cos (2 \pi a_k \omega)| \leq |\theta|$, this implies that for
every $k \in \N$, 
\begin{equation}\label{taylor_err}
1 - \frac{e^{|\theta|} |\theta|^{d+1}}{(d+1)!} \leq \frac{p(\theta \cos (2 \pi a_k \omega))}{e^{\theta \cos (2 \pi a_k \omega)}} \leq 1 + \frac{e^{|\theta|}|\theta|^{d+1}}{(d+1)!}.
\end{equation}
Assuming that  $d =d(\ve) \in\N$ is sufficiently large such that $\frac{e^{|\theta|}|\theta|^{d+1}}{(d+1)!} < \ve$, we have
$d(\ve)\to\infty$ as $\eps\to 0$, and  \eqref{eps_approx} follows. 
\end{proof}

As an immediate corollary we see that  for every $\ve > 0$, we have $d = d(\ve) \in \N$ such that for every $n \in \N$, $p = p_d$ satisfies 
\begin{equation}\label{ineq:integrals}
(1-\ve)^n \leq \frac{\int_0^1 \prod_{k=1}^n p \left(\theta \cos (2 \pi a_k \omega) \right) d\omega}{\int_0^1 \prod_{k=1}^n e^{\theta \cos (2 \pi a_k \omega)} d\omega} \leq (1+\ve)^n.
\end{equation}
Let $k_0=k_0(d)$ be a positive integer such that $a_{k+1}/a_k > 2d$ for all $k \geq k_0$; such an index must exist since we assumed that $a_{k+1}/a_k \to \infty$ as $k \to \infty$. For $n > k_0,$ 
we split
\begin{equation*} 
\prod_{k=1}^n p \left(\theta \cos (2 \pi a_k \omega) \right) = \left(\prod_{k=1}^{k_0} p \left(\theta \cos (2 \pi a_k \omega) \right)\right) \left(\prod_{k=k_0 + 1}^n p \left(\theta \cos (2 \pi a_k \omega) \right)\right), 
\end{equation*} 
and, taking into account that $p(\theta \cos (2 \pi a_k \omega))>0$ by~\eqref{eps_approx}, note that 
  \begin{equation}
    \label{prod_approx}
  m_0   \left(\prod_{k=k_0 + 1}^n p \left(\theta \cos (2 \pi a_k \omega) \right)\right)
  \leq \prod_{k=1}^n p \left(\theta \cos (2 \pi a_k \omega) \right) \leq 
  M_0 \left(\prod_{k=k_0 + 1}^n p \left(\theta \cos (2 \pi a_k \omega) \right)\right),
  \end{equation}
  where
  \[
  m_0 := \inf_{\omega \in [0,1]} \left( \prod_{k=1}^{k_0} p \left(\theta \cos (2 \pi a_k \omega) \right)\right)
  \quad \mbox{ and } \quad M_0 := \sup_{\omega \in [0,1]} \left(  \prod_{k=1}^{k_0} p \left(\theta \cos (2 \pi a_k \omega) \right) \right). 
  \]
We now state an important estimate on the integral of the common 
product that is on both  sides of the inequality
\eqref{prod_approx}.

\begin{lemma}
  \label{lem-estimate}
  Fix  $d \in \N$ and $p = p_d$ as in \eqref{def-poly}. Then for any
  $\theta, x \in \R$,
  \begin{equation}
    \label{exp_pk}
    p_d (\theta \cos x) = \sum_{j=0}^d b_j(\theta) \cos (jx),
  \end{equation}
  where the coefficients $b_j(\theta) = b_j(\theta;d), j = 0, 1, \ldots, d$, are  real numbers with 
   \begin{equation}
     \label{eqn:b_0_theta} 
      b_0 (\theta) =  b_0(\theta; d)  = \sum_{0 \leq m \leq \lfloor d/2 \rfloor} \frac{\theta^{2m}}{2^{2m} (m!)^2},
   \end{equation}
   and  for $j =1, \ldots,d$,  $b_j(\theta) \geq 0$ when $\theta> 0$.  
    Furthermore, given 
  $k_0 = k_0(d)   \in \N$ as above,  for every 
  $\theta \in \R$, and all $n > k_0$, 
   \begin{equation}
    \label{eq:int_prod_p}
    \int_0^1 \prod_{k=k_0 + 1}^n p \left(\theta \cos (2 \pi a_k \omega) \right) d\omega = b_0(\theta)^{n-k_0}. 
  \end{equation}
 \end{lemma}

  We first show how Theorem \ref{th1} follows from Lemma \ref{lem-estimate}, and then provide the
  proof of the lemma. 
  Integrating each term in the inequalities  in \eqref{prod_approx} with respect to Lebesgue measure over
  the interval $[0,1]$,  and applying~\eqref{eq:int_prod_p}, we obtain
$$
m_0  b_0(\theta)^{n-k_0} \leq \int_0^1 \prod_{k=1}^n p \left(\theta \cos (2 \pi a_k \omega) \right) d\omega \leq M_0  b_0(\theta)^{n-k_0}.
$$
Combining these inequalities with \eqref{ineq:integrals} and \eqref{mgf-sn}, we arrive at
$$
\frac{1}{(1+\ve)^n} m_0\, (b_0(\theta))^{n-k_0} \leq \E \left[ e^{\theta  S_n} \right] \leq \frac{1}{(1-\ve)^n} M_0\, (b_0(\theta))^{n-k_0}.
$$
Taking the natural logarithm of each term, dividing by $n$ and 
letting $n\to\infty$, while keeping all other variables fixed, we obtain
$$
\log b_0 (\theta) - \log (1+\eps)
\leq
\liminf_{n \to \infty} \frac{1}{n} \log \E \left[ e^{\theta S_n} \right]
\leq
\limsup_{n \to \infty} \frac{1}{n} \log \E \left[ e^{\theta S_n} \right]
\leq
\log b_0(\theta)- \log (1-\ve).
$$

Recall that $\ve\in (0,1)$ was fixed but arbitrary, that $d=d(\ve)$ depends on $\ve$ and satisfies $d(\ve) \rightarrow \infty$ as $\ve \rightarrow 0$,
  and that $b_0(\theta) = b_0(\theta;d(\ve))$ depends on this choice of $d$. Note that $b_0(\theta)$ is a finite partial sum of the series expansion for the moment generating function of the arcsine distribution on the interval $(-1,1)$, which we derived in~\eqref{arcsin_mgf}.
Since the logarithm is a continuous function, \eqref{eqn:b_0_theta} and~\eqref{arcsin_cumgf} yield
$$
\lim_{\eps\to 0} \log b_0(\theta; d(\ve)) = \widetilde{\Lambda}(\theta).
$$
Thus, the last two displays together yield the limit 
$$
\lim_{n \to \infty} \frac{1}{n} \log \E \left[ e^{\theta S_n} \right] = \widetilde{\Lambda}(\theta),
$$
as desired. Since the function $\widetilde{\Lambda}$ is differentiable, the G\"artner-Ellis theorem can be applied, and the proof of Theorem \ref{th1} is complete, given Lemma \ref{lem-estimate}.

  To complete the proof of Theorem \ref{th1}, it only remains to establish Lemma \ref{lem-estimate}. 

  \begin{proof}[Proof of Lemma \ref{lem-estimate}] 
  For every fixed $k \in \N$, the function $\omega \mapsto$ $p \left(\theta \cos (2 \pi a_k \omega) \right)$ is a polynomial of degree $d$ in $\theta \cos (2 \pi a_k \omega)$. A standard trigonometric formula asserts that for even $m = 2n$,  $(\cos x)^m$ can be expressed as a linear combination of $1, \cos (2x), \cos (4x), \dots, \cos (mx)$, more precisely, for every $n\in\N$ and $x \in \R$, 
\begin{equation}\label{prod_cos_1}
(\cos x)^{2n} = \frac{1}{2^{2n}}{2n\choose n} + \frac{1}{2^{2n-1}}\sum_{\ell=0}^{n-1} {2n\choose \ell} \cos(2(n-\ell)x).
\end{equation}
For odd $m$, $(\cos x)^m$ can be expressed as a linear combination of $\cos x, \cos (3x), \ldots, \cos (mx)$, more precisely, for every $n\in\N$ and  $x \in \R$, 
\begin{equation}\label{prod_cos_2}
(\cos x)^{2n+1} = \frac{1}{4^{n}}\sum_{\ell=0}^{n} {2n+1\choose \ell} \cos((2n+1-2\ell)x).
\end{equation}
The precise statements of the last two formulas will not be important to us; we will only use the fact that the coefficient
of the constant term in the expansion of $(\cos x)^m$ is $\binom{m}{m/2} 2^{-m}$ when $m$ is even and zero otherwise. So for $d\in\N$ and  $\theta \in \R$, 
it is possible to write $p\left(\theta \cos x \right)$ for $x \in \R$, of the form
\[ p \left(\theta \cos x \right) = 
  b_0(\theta) + b_1(\theta)  \cos x + b_2 (\theta)\cos (2 x) + \dots + b_d(\theta)  \cos(d x),
\]
where the coefficients $b_k(\theta) = b_k(\theta;d), k=1, \ldots, d,$ depend on the coefficients of the polynomial $p$ (and thus on $d$) and on $\theta$, but not 
on $x$, and the zeroth coefficient $b_0(\theta)$ takes the explicit form  
\begin{eqnarray}\label{eq:b_0_theta}
b_0 (\theta)
& = & \sum_{\substack{0 \leq m \leq d,\\m \text{~even}}} \binom{m}{m/2} \frac{\theta^m}{2^{m} m!} \notag\\
& = & \sum_{0 \leq m \leq \lfloor d/2 \rfloor} \binom{2m}{m} \frac{\theta^{2m}}{2^{2m} (2m)!} \notag\\
& = & \sum_{0 \leq m \leq \lfloor d/2 \rfloor} \frac{(2m)!}{m! m!} \frac{\theta^{2m}}{2^{2m} (2m)!} \notag\\
& = & \sum_{0 \leq m \leq \lfloor d/2 \rfloor} \frac{\theta^{2m}}{2^{2m} (m!)^2}, 
\end{eqnarray}
which agrees with  \eqref{eqn:b_0_theta}.  This proves \eqref{exp_pk}.
  Further, when  $\theta  > 0$, since  the Taylor coefficients of the exponential function are all positive, and  
  the coefficients in the trigonometric identities \eqref{standard_id} and \eqref{prod_cos_1}
  are all non-negative, it  follows that  $b_j(\theta) \geq 0$ for  $j =1,\ldots,d$.

It  only remains to show  that when  integrating the product-form integrand  on the right-hand side of
  \eqref{eq:int_prod_p}  only  terms involving the zeroth coefficient remain. 
To this end, note that by \eqref{exp_pk} we have for $\omega \in [0,1]$, 
\begin{eqnarray*}
& & \prod_{k=k_0 + 1}^n p \left(\theta \cos (2 \pi a_k \omega) \right) \\
& = & \prod_{k=k_0 + 1}^n \Big(b_0 (\theta) + b_1 (\theta) \cos (2 \pi a_k \omega) + b_2 (\theta) \cos (2 \pi 2 a_k \omega) + \dots + b_d (\theta) \cos( 2 \pi d a_k \omega) \Big).
\end{eqnarray*}
When multiplying out this product, we obtain a constant term $b_0(\theta)^{n-k_0}$ as well as a sum of many mixed terms of the form
$$
b_0(\theta)^{n-k_0-\ell} \cdot  b_{j_1}(\theta) \cos \left( 2 \pi  j_1 a_{k_1} \omega \right) \cdot \ldots \cdot b_{j_\ell}(\theta) \cos \left(2 \pi j_\ell a_{k_\ell} \omega\right),
$$
for some $\ell \in \{1, \dots, n - k_0\}$,  $(j_1, \dots, j_\ell) \in \{1, \dots, d\}^\ell$, and  $(k_1, \dots, k_\ell) \in \{k_0+1, \dots, n\}^\ell$ such that $k_1 > \dots > k_\ell$. Thus, to prove 
  \eqref{eq:int_prod_p}, it suffices to show that 
  for any such configuration, we have 
\begin{equation}
  \label{toshow1}
\int_0^1  \cos (2 \pi j_1 a_{k_1} \omega) \cdot \ldots \cdot \cos (2 \pi j_\ell a_{k_\ell} \omega) d\omega = 0.
\end{equation}
We now show that this follows because  $a_{k+1}/a_k>2d$ for $k \geq k_0$ (by the choice of $k_0$) Indeed, recall the standard trigonometric identity
\begin{equation}
  \label{standard_id}
  \cos x \cos y = \frac{1}{2} \Big( \cos(x-y) + \cos(x+y) \Big), 
\end{equation}
which implies that 
the product $\cos (2 \pi j_1 a_{k_1} \omega) \cdot \ldots \cdot \cos (2 \pi j_\ell a_{k_\ell} \omega)$ can be written as a linear combination of cosine functions $\cos (2\pi m \omega)$ with frequencies of the form
$$
m = j_1 a_{k_1} \pm \dots \pm j_\ell a_{k_\ell}.
$$
As already mentioned above, we have  $k_1>k_2>\ldots >k_\ell >k_0$. Then, taking into account that $j_1\geq 1$, we have 
$$
j_1 a_{k_1} \pm \dots \pm j_\ell a_{k_\ell} \geq a_{k_1} - d a_{k_2} - d a_{k_3} - \dots - d a_{k_\ell}.
$$
The inequality  $a_{k+1}/a_k > 2d$ for all $k \geq k_0$ then implies 
$$
a_{k_1} - d a_{k_2} - d a_{k_3} - \dots - d a_{k_\ell} \geq a_{k_1} \left( 1 - d \sum_{r=1}^\ell \frac{1}{(d+1)^r} \right) =
a_{k_1} \underbrace{ \left( 1 - d \left(\frac{1 - (d+1)^{-\ell}}{d} \right)\right)}_{>0} > 0. 
$$
Consequently,  the product $\cos (2 \pi j_1 a_{k_1} \omega) \cdot \ldots \cdot \cos (2 \pi j_\ell a_{k_\ell} \omega)$ can be written as a linear combination of cosine functions $\cos (2\pi m \omega)$ that have \emph{all non-zero} frequencies $m\in \N$.
This clearly implies \eqref{toshow1}, and thus completes the proof.
\end{proof}

\subsection{Proofs of Theorem \ref{th2} and Lemma \ref{lem:I_q_+1}}\label{subs-pf-theoremB}

Let $q \in \{2,3,\ldots\}$ be fixed,  let $a_k = q^k$ for each $k\in\N$, and let
  $S_n$ be as defined in \eqref{def-sn}. We establish the LDP by first recasting  $S_n/n$ 
  as a Birkhoff average (or time average) of a stationary sequence induced by the expanding piecewise continuous map
  $\mapT:[0,1]\to [0,1]$ given by
  \begin{equation}
    \label{def-mapT}
    \mapT \omega :=  q \omega \pmod 1 =q\omega - \lfloor q\omega \rfloor , \qquad \omega \in [0,1],
  \end{equation}
 (which is merely the fractional part of $q\omega$). 
Then,  \eqref{def-sn} and the identity $a_k = q^k$ show that the lacunary sums of interest can  be expressed as 
\begin{equation}
  \label{sn-Birkhoff}
  S_n(\omega) = \sum_{k=0}^{n-1} X_1(\mapT^{k}\omega) = X_1(\omega)+\ldots +X_n(\omega), \quad \omega \in [0,1]. 
  \end{equation}
  We can then  apply  tools
  from the theory of LDPs for  (uniform and non-uniform) hyperbolic dynamics and mixing processes;
  see for example~\cite{lopes,orey_pelikan,kifer,young,grigull,kesseboehmer,broise,collet,bryc,denker_kesseboehmer,denker_nicol}.
  Since in some references (see, e.g., ~\cite[p.~422]{cc}, \cite[Thm.~10.8 on p.~90]{broise}), an LDP is
  stated  only for 
some small neighborhood of $0$, and since parts of the argument will be needed to prove
property (iii) in the statement of Theorem \ref{th2}, we provide a sketch of the full proof in Section \ref{subs-pf-th2LDP}. 
The proofs of properties (i)--(iv), which are the main thrust of Theorem \ref{th2}, are presented
in Section \ref{subs-pf-th2props}. They  rely on additional estimates that are first obtained in Section \ref{subs-pf-th2estimates}.
Finally, the proof of Lemma \ref{lem:I_q_+1} is given in Section \ref{subs-pf-lemIq}.

\subsubsection{Proof of the LDP in Theorem \ref{th2}} \label{subs-pf-th2LDP}

By the G\"artner-Ellis Theorem, to prove the LDP it suffices to show that the limit $\Lambda_q (\theta) := \lim_{n \to \infty} \frac{1}{n} \log \E [ e^{\theta S_n}]$ exists for all $\theta\in\R$ and is differentiable in $\theta$.
We now express $e^{\theta S_n}$ in terms of  a certain linear operator.
 Let $\text{Lip}[0,1]$ denote the Banach space of Lipschitz functions 
  $f:[0,1]\to \mathbb C$,  endowed with the norm $\|f\| := \|f\|_\infty + L(f)$, where $L(f)$
 is the Lipschitz constant of $f$.
 Next, for $\theta\in\R$, consider the linear operator 
 $\Phi_{\theta, q}: \text{Lip}[0,1]\to \text{Lip}[0,1]$
 defined, for $g\in \text{Lip}[0,1]$, by 
\begin{equation}\label{PerFrobPerturb0}
(\Phi_{\theta, q} g) (\omega)
 :=\frac 1q \sum_{j=0}^{q-1} e^{\theta X_1\left(\frac{\omega + j}{q}\right)} g\left(\frac{\omega + j}{q}\right),
\qquad \omega \in [0,1], 
\end{equation}
where we recall from \eqref{def-xk} that   $X_1(\omega) = \cos (2\pi q \omega)$, 
$\omega\in[0,1]$.

The proof of the LDP  for $(S_n/n)_{n\in\N}$ stated in Theorem \ref{th2} is a direct consequence of the following proposition.  
    
    \begin{prop}
      \label{prop-mgf-PF}
      Fix $q \in \{2,3, \ldots\}$ and $\theta \in \R$. Then 
      \[  \Lambda_q(\theta) := \lim_{n \to \infty} \frac{1}{n} \log \E [ e^{\theta S_n}] = \log \lambda_{\theta,q}, \]
      where $\lambda_{\theta,q}$ is the Perron-Frobenius eigenvalue of the operator  $\Phi_{\theta, q}$ defined
      in \eqref{PerFrobPerturb0}.
      Moreover, there exists an open domain ${\mathcal D}$ of the complex plane that contains the real line $\mathbb{R}$  such that the  convergence above holds  uniformly for $\theta$ in any compact subset of ${\mathcal D}$.  
      In particular, $\theta \mapsto \Lambda_q(\theta)$ is differentiable. 
    \end{prop}

In the language of thermodynamic formalism~\cite{zinsmeister}, $\log \lambda_\theta$ is referred to as the pressure or the free energy of a one-dimensional lattice system, and its differentiability  expresses the known fact that there are no phase transitions for such systems. (For more background on the  spectral gap property of  Perron-Frobenius transfer operators,  the reader is referred to ~\cite{baladi_book}, \cite{broise}, \cite{rychlik} \cite[Chapter~4]{zinsmeister}.)

\begin{proof}[Proof of Proposition \ref{prop-mgf-PF}.]
Recall the definition of the  map $\mapT:[0,1] \to [0,1]$ given in \eqref{def-mapT} and note that Lebesgue measure is an invariant measure for $\mapT$, i.e., $\mapT$ maps the  measure space $([0,1],\mathcal B([0,1]),\lambda)$ to itself and satisfies $\lambda({\mathcal T}^{-1}(A))  = \lambda(A)$ for every $A \in {\mathcal B}[0,1]$.  
Indeed, (for simplicity we only consider $q=2$)
  \[
  \mapT(x) = 2x\pmod1 = 2x - \lfloor 2x\rfloor =
 \begin{cases}
   2x &:\,0\leq x < \frac{1}{2} \\
    2x-1 &:\, \frac{1}{2}\leq x \leq 1
 \end{cases}
 \]
  and so, for every positive and measurable function $f:[0,1]\to \R$,
  \begin{align*}
 \int_{[0,1]}f(\mapT(x))\,\lambda(dx) & = \int_{[0,1/2)}f(2x)\,\lambda(dx) + \int_{[1/2,1]}f(2x-1)\,\lambda(dx) \cr
 & = \frac{1}{2}\int_{[0,1]}f(y)\,\lambda(dy) + \frac{1}{2}\int_{[0,1]}f(y)\,\lambda(dy) = \int_{[0,1]}f(y)\,\lambda(dy).
  \end{align*}
      In ergodic theory parlance, $\big(([0,1],\mathcal B([0,1]),\lambda);\mapT\big)$ is a measure-preserving
      dynamical system and we refer the reader to \cite{EFHN2015} for further details.  
      The Perron-Frobenius operator $\Phi_q: \text{Lip}[0,1]\to \text{Lip}[0,1]$ associated with $\mapT$ is defined by 
$$ 
(\Phi_q g) (\omega) = \frac 1q \sum_{j=0}^{q-1}  g\left(\frac{\omega + j}{q}\right),
\qquad \omega \in [0,1], 
g\in \text{Lip}[0,1],
$$
where recall $\text{Lip}[0,1]$ is the space of Lipschitz functions defined above.   Note that for any $g\in\text{Lip}[0,1]$ 
\[ \int_0^1 (\Phi_q g) (\omega) d \lambda (\omega) 
=  \frac 1q  \sum_{j=0}^{q-1} \int_0^1 g\left(\frac{\omega + j}{q} \right) d \lambda (\omega) 
= \int_0^1 g(\omega) d \lambda (\omega),   
\]
where the last equality uses the fact that $\mapT$  is $\lambda$-preserving and 
     $\{(\omega + j)/q,$  $j = 0, 1, \ldots, q-1\}$ is  the preimage of $\omega$ under $\mapT$. 
This shows that $\Phi_q$ preserves the integral for any function  $g\in\text{Lip}[0,1]$.
Next, for $\theta\in\R$, note that the operator $\Phi_{\theta, q}$
  defined in \eqref{PerFrobPerturb0} can be viewed
  as a perturbation  of  the operator $\Phi_q$   since for $g\in \text{Lip}[0,1]$, 
\begin{equation}\label{PerFrobPerturb}
(\Phi_{\theta, q} g) (\omega) = 
\left(\Phi_q [e^{\theta X_1} g]\right) (\omega)
=
\frac 1q \sum_{j=0}^{q-1} e^{\theta X_1\left(\frac{\omega + j}{q}\right)} g\left(\frac{\omega + j}{q}\right),
\qquad \omega \in [0,1], 
\end{equation}
where once again recall from \eqref{def-xk} that $X_1(\omega) = \cos (2\pi q \omega)$, 
$\omega\in[0,1]$.  It is immediate from the definition that both $\Phi_q$
  and $\Phi_{\theta, q}$ are linear operators. Denoting by $\Phi_q^n$ and $\Phi_{\theta,q}^n$ the
  $n$-fold composition of $\Phi_q$ and $\Phi_{\theta, q}$, respectively, a straightforward inductive argument  (see, e.g., 
 \cite[Proposition 5.1 (P4)]{broise}), shows that
 \begin{equation}
   \label{Phi-thetan}
\Phi_{\theta, q}^n g = \Phi_q^n [e^{\theta S_n} g], \quad \mbox{ for every } n\in\N. 
\end{equation}
Let  $\mathbf{1}$ denote the constant function on $[0,1]$ that takes the value $1$, and henceforth, denote $d\lambda(\omega)$ just as $d\omega$.  Then,
   the  moment generating function of $S_n$ can be expressed as
\begin{align}\label{eq:expectation integral perturbation operator}
\E\big[e^{\theta S_n}\big] = \int_{0}^1 e^{\theta S_n(\omega)} d \omega
=
\int_{0}^1 \Phi_q^n [e^{\theta S_n}](\omega) d \omega
=
\int_{0}^1 \left(\Phi_{\theta,q}^n \mathbf{1}\right) (\omega) d \omega, 
\end{align}
where the second equality uses the fact that
  $\Phi^n_q$ preserves the integral and the last equality uses
  \eqref{Phi-thetan} with $g = \mathbf{1}$.

We will now use 
the crucial fact is that the operator $\Phi_{\theta, q}$ has the spectral gap property; see, e.g., \cite[Theorems~4.1 and~4.23]{zinsmeister} and~\cite[Theorem~1.5]{baladi_book}, where all essential arguments can be found.
Namely, we use the well-known fact that for every $\theta\in\R$, 
$\Phi_{\theta, q}$ admits a decomposition 
\begin{equation}\label{eq:dec_Phi_theta}
\Phi_{\theta, q}  = \lambda_\theta Q_\theta + R_\theta,
\end{equation}
where
\begin{equation}
  \label{deponq}
 \lambda_\theta  = \lambda_{\theta,q} >0
\end{equation}
is a simple eigenvalue of $\Phi_{\theta, q}$, $Q_\theta = Q_{\theta,q}$ is a projection operator onto a line spanned by an eigenfunction $h_\theta = h_{\theta,q}>0$ associated with
$\lambda_\theta$, and $R_\theta = R_{\theta,q}$ is an operator whose spectral radius is strictly smaller than $\lambda_\theta$.  More precisely, for $f \in \text{Lip}[0,1]$, 
there is a probability measure $\mu_\theta = \mu_{\theta,q}$ on $[0,1]$ such that
$$
Q_\theta f = h_\theta \cdot \frac{\int_0^1 f(\omega) d\mu_\theta(\omega)}{\int_0^1 h_\theta (\omega) d\mu_\theta(\omega)}
\quad  
\text{ and }
\quad
R_\theta Q_\theta  =   Q_\theta R_\theta \equiv 0\, .
$$
Continuing to omit   the dependence of  the quantities  $\lambda_\theta, R_\theta, Q_\theta, h_\theta$ and $\mu_\theta$ on $q$,  
by raising the decomposition of $\Phi_{\theta, q}$ to the $n$-th power, it follows that for any
$f \in \text{Lip}[0,1]$,  
$$
\Phi_{\theta, q}^n f
=
\lambda_\theta^n Q_\theta^n f + R_\theta^n f
=
\lambda_\theta^n \cdot h_\theta \cdot \frac{\int_0^1 f (\omega)d\mu_\theta(\omega)}{\int_0^1 h_\theta(\omega) d\mu_\theta(\omega)}
+
R_\theta^n f.
$$
Now, setting $f = \mathbf{1}$,   
taking the integral on both sides, and using \eqref{eq:expectation integral perturbation operator},
one obtains 
\begin{align}\label{eq: expectation decomposition}
\E \big[e^{\theta S_n}\big]
& =
\int_{0}^1 \left(\Phi_{\theta,q}^n \mathbf{1} \right) (\omega) d \omega
=
\lambda_\theta^n \cdot \frac{\int_0^1 h_\theta(\omega) d\omega}{\int_0^1
  h_\theta(\omega) d\mu_\theta(\omega)}
+
\int_{0}^1 (R_\theta^n \mathbf{1}) (\omega) d\omega.
\end{align}
Recalling that the spectral radius of $R_\theta$ is strictly smaller than $\lambda_\theta$,
one obtains
\begin{equation}\label{eq:mod_phi}
\lim_{n\to\infty}
\frac {\E \big[e^{\theta S_n}\big]}{\lambda_\theta^n}
=
\frac{\int_0^1 h_\theta(\omega) d\omega}{\int_0^1
  h_\theta (\omega) d\mu_\theta(\omega)}.   
\end{equation}
Note in passing that this shows that the sequence $(S_n)_{n\in\N}$ satisfies some version of  mod-phi convergence~\cite{meliot_etal_book}, but what is more pertinent, it implies 
the weaker statement
\begin{equation}\label{eq:mod_phi_weak}
\lim_{n\to\infty} \frac 1n \log \E \big[e^{\theta S_n}\big] = \log \lambda_\theta, 
\end{equation}
which proves the first assertion of the proposition. 

We now turn to the proof of the remaining assertions, 
which we claim (and justify below) can be deduced from the perturbation theory of linear operators \cite[Chapter 7, $\S$3, p. 368]{Kato_book}, in particular the Kato-Rellich theorem, as stated in~\cite[Theorem~4.24]{zinsmeister}. Indeed, since the family of operators $\Phi_{\theta, q}$ depends on $\theta\in\mathbb C$ in an analytic way (see \cite[Proposition 5.1 (P3)]{broise} and \cite[Theorem 1.7, p.368]{Kato_book}), the decomposition~\eqref{eq:dec_Phi_theta} continues to hold in some neighborhood $\mathcal D$ of the real axis (with $\lambda_\theta$, $h_\theta$ and $\mu_\theta$ becoming complex-valued), with $\lambda_\theta\neq 0$
and $\lambda_\theta$ (as well as $h_\theta,\mu_\theta, R_\theta$) being analytic on $\mathcal D$. Moreover, $|\lambda_\theta|$ stays strictly smaller than the spectral radius of $R_\theta$ if $\mathcal D$ is sufficiently small, which, looking at \eqref{eq: expectation decomposition}, shows that convergence in \eqref{eq:mod_phi_weak} is uniform on compact subsets of $\mathcal D$.
\end{proof}

\subsubsection{Moment estimates for the partial sums $S_n$ and $\widetilde{S}_n$}
\label{subs-pf-th2estimates}

Let $n\in\N$ and consider
\begin{align}\label{eq: expression Sn}
\Lambda_{q,n}(\theta):=\frac{1}{n} \log \E\big[e^{\theta S_n}\big] = \frac{1}{n} \log \sum_{m=0}^\infty \frac{\theta^m}{m!}\E\big[S_n^m\big],
\end{align}
and
\begin{align}\label{eq: expression tilde Sn}
\widetilde{\Lambda}_n(\theta):=\frac{1}{n} \log \E\big[e^{\theta \widetilde{S}_n}\big] = \frac{1}{n} \log \sum_{m=0}^\infty \frac{\theta^m}{m!}\E\big[\widetilde{S}_n^m\big],
\end{align}
where we recall that $\widetilde{S}_n=\sum_{j=1}^n \widetilde{X}_j$, and 
$(\widetilde{X}_j)_{j \in \N}$ are i.i.d.~having the
same distribution as $X_1$, as defined in \eqref{def-wtildesn} and \eqref{def-wtildexk}, respectively.
The proof of properties (i)--(iv) in Theorem \ref{th2}, presented in the next section,
   involve  a comparison of the coefficients in the Taylor expansions of \eqref{eq: expression Sn} and \eqref{eq: expression tilde Sn} considered as functions of $\theta$, which  in turn relies on estimates 
  on  the moments of $S_n$ and $\widetilde{S}_n$, obtained in 
Lemmas \ref{lem:A_m_n_moment}--\ref{lem:arcsine_moments} below. 
  We start with  Lemma \ref{lem:A_m_n_moment} on estimates of the  moments of $S_n$. 

\begin{lemma}\label{lem:A_m_n_moment}
  Fix $q\in \{2,3,\ldots\}$,  let $a_k = q^k$ for all $k\in\N$, and let $S_n$ be as defined in \eqref{def-sn}.
  Then, for every $m, n \in\N$, we have
$$
\E[S_n^m] = \frac{A_m(n)}{2^m},
$$
where $A_m(n)$ is the number of solutions to the equation $\sum_{i=1}^m \eps_i q^{k_i} = 0$
in the unknowns $k_1,\ldots,k_m\in \{1,\ldots,n\}$ and $\eps_1,\ldots,\eps_m \in \{+1,-1\}$.
\end{lemma}
\begin{proof}
For every $m\in\N$, we have
\begin{align*}
\E\big[S_n^m\big] & = \int_0^1 \Big(\sum_{k=1}^n \cos(2\pi q^k \omega)\Big)^m\,d\omega \\
& = \int_0^1 \sum_{k_1,\dots,k_m=1}^n\, \prod_{\ell=1}^{m}\cos(2\pi q^{k_\ell}\omega)\,d\omega \\
& = \frac{1}{2^m}\int_0^1 \sum_{k_1,\dots,k_m=1}^n\, \prod_{\ell=1}^{m} \Big[e(q^{k_\ell}\omega)+e(-q^{k_\ell}\omega)\Big] \,d\omega,
\end{align*}
where we write $e(z) := e^{2\pi i z}$ for $z\in\mathbb R$ and used that $\cos z = (e^{iz}+e^{-iz})/2$. By rewriting the product in the last line of the last display in terms of an exponential and using the elementary identity $\int_0^1 e(k \omega) d \omega = 0$ for all integer $k\neq 0$, we see that
\begin{align}\label{moments s_n}
\E\big[S_n^m\big] = \frac{1}{2^m} \sum_{k_1,\dots,k_m=1}^n\,\sum_{\varepsilon_1,\dots,\varepsilon_m\in\{-1,1\}} \mathbbm 1_{\big\{\varepsilon_1q^{k_1}+\cdots+\varepsilon_mq^{k_m}=0 \big\}}.
\end{align}
To complete the proof of the lemma, observe that the right-hand side equals $A_m(n)/2^m$.
\end{proof}

Next, we give a combinatorial interpretation of $A_m(n)$ for $m\leq q$. 
Let $B_m(n)$ be  the number of simple random walk paths in $\mathbb Z^n$ of length $m$ that return to the origin, which is sometimes
also referred to as 
the number of bridges of length $m$ in $\mathbb Z^n$.

\begin{lemma}\label{lem:A_m_n_equals_B_m_n}
 For all  $n, m\in\N$, we have $A_m(n) \geq B_m(n)$ and, if $m \leq q$, then 
  $A_m(n) = B_m(n)$.  
\end{lemma}
\begin{proof}
We start with the proof of the second statement. Let $m\leq q$. We first claim (and justify below) that 
\begin{equation}\label{eq:syst_0}
\sum_{\ell=1}^m \varepsilon_\ell q^{k_\ell} = 0, \qquad k_\ell\in\{1,\dots, n\},\, \varepsilon_\ell \in\{-1,1\}
\end{equation}
if and only if for every $k\in\{1,\dots,n\}$,
\begin{align}\label{eq:condition on H_k}
H_k = \sum_{\ell=1}^m \mathbbm 1_{\{k_\ell=k \}}\cdot \varepsilon_\ell = 0.
\end{align}
In other words, \eqref{eq:syst_0} can hold only if every term $+q^k$ is canceled by a term $-q^k$ at some other place.
One direction of the claim is immediate. We note that
\begin{equation}
  \label{identity}
\sum_{\ell=1}^m \varepsilon_\ell q^{k_\ell} = \sum_{k=1}^n \Big(\sum_{\ell=1}^m \mathbbm 1_{\{k_\ell = k \}}\cdot\varepsilon_\ell\Big)q^{k} = \sum_{k=1}^n H_k\,q^{k}
\end{equation}
and therefore if all $H_k$ vanish, then $\sum_{\ell=1}^m \varepsilon_\ell q^{k_\ell} = 0$. 
For the opposite direction, suppose $\sum_{\ell=1}^m \varepsilon_\ell q^{k_\ell} = 0$.  Then, due to the
identity in \eqref{identity}, 
\[
\sum_{k=1}^n H_k\,q^{k} = 0. 
\] 
We first show that this, along with the fact that $m \leq q$, implies $H_1 = 0$. 
First, dividing everything by $q\geq 2$, we obtain
\begin{align}\label{eq: equation divided by q}
\sum_{k=1}^n H_k\,q^{k-1} = 0, 
\end{align}
which clearly  implies divisibility of $H_1$ by $q$.
Now, if $m<q$, then $|H_1|\leq m < q$ by definition. Hence, $H_1=0$. If $m=q$, then either $H_1=0$ or $H_1=\pm q$ and the latter case only occurs if all $\varepsilon_1,\dots,\varepsilon_m$ are equal and $k_\ell=1$ for all $\ell\in\{1,\dots,m\}$. In this case, the condition $\sum_{\ell=1}^m \varepsilon_\ell q^{k_\ell} = 0$ is violated. Hence, for $m\leq q$, we have $H_1=0$. Now dividing \eqref{eq: equation divided by q} by $q$ and repeating the argument, it follows that $H_2 = \cdots = H_n=0$ as well.
This completes the proof of the claim of  equivalence between the conditions \eqref{eq:condition on H_k} and
\eqref{eq:syst_0}. 

Next, note that the conditions \eqref{eq:condition on H_k} on $H_k$ may be interpreted as follows: for given $\varepsilon_1,\dots,\varepsilon_m\in\{-1,1\}$ and $k_1,\dots,k_m\in\{1,\dots,n\}$, we consider the nearest neighbor path of length $m$ in $\mathbb Z^n$ whose $\ell^{\text{th}}$ step is equal to $\varepsilon_\ell \vec{e}_{k_\ell}$ for $\vec{e}_1,\dots,\vec{e}_n$
the standard vector basis in $\R^n$. Clearly, condition \eqref{eq:condition on H_k} is satisfied if and only if the path
   returns to its starting point. It follows that $A_m(n) = B_m(n)$, which proves the second assertion of the lemma.

\vskip 2mm
To prove the first assertion, note that if $m\in\N$ is arbitrary, then the solutions of~\eqref{eq:syst_0} can be divided into the trivial ones (i.e.,\ those for which $H_1=\ldots=H_n = 0$), and the non-trivial ones (such as $q^2 - q - \ldots - q = 0$ for $m=q+1$, where the term $-q$ appears $m$ times). Since the number of trivial solutions is $B_m(n)$, and (by definition)  
$A_m(n)$ is the total number of solutions, the claim $A_m(n) \geq B_m(n)$ follows.
\end{proof}

\vspace*{2mm}
Taken together, Lemmas~\ref{lem:A_m_n_moment} and~\ref{lem:A_m_n_equals_B_m_n} show that, for each $m\leq q$,
\[
\E\big[S_n^m\big] = \frac{B_{m}(n)}{2^m}.
\]
Let us turn to the computation of $\E[\widetilde{S}_n^m]$, where we shall prove that the analogous identity holds, this time \textit{for all} $m\in\N$.
\begin{lemma}\label{lem:arcsine_moments}
  Recall that $\widetilde S_n = \widetilde X_1+ \ldots +  \widetilde X_n$, where $\widetilde X_1,\widetilde X_2,\ldots$ are i.i.d.~random variables with the arcsine distribution on $(-1,1)$.  Then, for all $m, n\in \N,$  we have
\[
\E\big[\widetilde{S}_n^m\big] = \frac{B_{m}(n)}{2^m}.
\]
\end{lemma}
\begin{proof}
Recalling that $(U_k)_{k\in\N}$ is a sequence of i.i.d.~random variables with the same uniform distribution as $U$, we can write
\[
\widetilde{X}_k =  \cos (2\pi U_k) = \frac{e(U_k)+e(-U_k)}{2}, 
\]
where we again write $e(z) := e^{2\pi i z}$.
For $m\in\N$, we have
\begin{align*}
\E\big[\widetilde{S}_n^m\big] & = \frac{1}{2^m}\E \bigg[\left(\sum_{k=1}^n \big[e(U_k)+e(-U_k)\big]\right)^m\bigg] \\
& = \frac{1}{2^m} \sum_{k_1,\dots,k_m=1}^n \,\sum_{\varepsilon_1,\dots,\varepsilon_m\in\{-1,1\}} \E\bigg[\prod_{\ell=1}^me(\varepsilon_\ell U_{k_\ell}) \bigg] \\
& = \frac{1}{2^m} \sum_{k_1,\dots,k_m=1}^n \,\sum_{\varepsilon_1,\dots,\varepsilon_m\in\{-1,1\}} \E\bigg[ e\Big(\sum_{\ell=1}^m \varepsilon_\ell U_{k_\ell}\Big)\bigg] \\
& = \frac{1}{2^m} \sum_{k_1,\dots,k_m=1}^n \,\sum_{\varepsilon_1,\dots,\varepsilon_m\in\{-1,1\}} \E\bigg[ e\Big(\sum_{\ell=1}^n H_\ell U_{\ell}\Big)\bigg],
\end{align*}
where for any fixed  $k_i \in \{1, \ldots, n\}$, 
$\varepsilon_i \in \{-1,1\}$,  $i = 1, \ldots, m$, we set 
\[ H_\ell := \sum_{i=1}^m \mathbbm 1_{\{k_i= \ell \}}\cdot \varepsilon_\ell, \quad \ell = 1, \ldots. n. \]
Since we have
\[  \E\bigg[ e\Big(\sum_{\ell=1}^n H_\ell U_{\ell}\Big)\bigg] = \prod_{\ell=1}^n \E\bigg[ e\Big(H_\ell U_{\ell}\Big)\bigg]
= \sum_{\varepsilon_1,\dots,\varepsilon_m\in\{-1,1\}} \mathbbm 1_{\{ H_1=\dots=H_n=0\}}. \]
Since $H_1=\dots=H_n=0$ if and  only if the associated nearest neighbor path of length $m$ in $\mathbb{Z}^n$, 
whose $\ell^{th}$ step is equal to $\varepsilon_i \vec{e}_{k_i}$, with $\vec{e}_1, \ldots, \vec{e}_n$ the standard basis in
$\mathbb{Z}^n$, returns to its starting point, we have shown that
\[ \E\big[\widetilde{S}_n^m\big] = \frac{1}{2^m} \sum_{k_1,\dots,k_m=1}^n \,\sum_{\varepsilon_1,\dots,\varepsilon_m\in\{-1,1\}}
 \mathbbm 1_{\{ H_1=\ldots=H_n=0\}} =  \frac{B_{m}(n)}{2^m},
\]
which completes the proof.
\end{proof}

\subsubsection{Proof of Properties (i)--(iv) of Theorem \ref{th2}.}
\label{subs-pf-th2props} 
We now complete the proof of Theorem \ref{th2}.
First, note that the function $\Lambda_q$, as a uniform limit of analytic functions, is itself analytic for all $\theta \in \mathbb{C}$,
$|\theta| < \ve_0$, for a sufficiently small $\ve_0 > 0$.\\

\noindent 
{\em Proof of (i).}  First, let us observe that the proof that $I_q \leq \widetilde{I}$ on the positive real axis is simple.
Indeed, Lemmas \ref{lem:A_m_n_moment}--\ref{lem:arcsine_moments} imply that for all $m, n \in \N$,
\[ \E[S_n^m] \geq \E[ \widetilde{S}_n^m].
\]
When combined with \eqref{eq: expression Sn} and \eqref{eq: expression tilde Sn},  
  it follows that for every $n \in \N$ and $\theta > 0$,
\[ \Lambda_{q,n} (\theta) \geq \widetilde{\Lambda}_n (\theta). \]
Passing to the limit as $n \rightarrow \infty$ on both sides, and noting that both limits exist and and are
equal to $\Lambda_q (\theta)$ and $\widetilde{\Lambda}_q(\theta)$, respectively, due to the proof in Section \ref{subs-pf-th2LDP}
and the independence of $(\widetilde{X}_k)_{k \in \N}$, we conclude that $\Lambda_q(\theta) \geq \widetilde{\Lambda}_q(\theta)$ for all $\theta > 0$.
    Passing to the Legendre-Fenchel transform we then obtain $I_q(x) \leq \widetilde{I}(x)$ for all $x > 0$.

    The proof of the  strict inequality $I_q (x) < \widetilde{I}(x)$ for $x > 0$ is more delicate.  Assume that $q \geq 2$ and
      $\theta > 0$ are fixed.  We choose  a large integer $d > q$; at the end of the  proof we will  let $d \rightarrow \infty$. 
      As in the proof of Theorem \ref{th1}, we approximate the exponential function by a Taylor polynomial $p = p_d$ of degree $d$, and by
      \eqref{taylor_err}, we have
    \[ 
     \E \left[e^{\theta S_n} \right] \geq \left( 1 + \frac{e^\theta \theta^{d+1}}{(d+1)!} \right)^{-n}\int_0^1 \prod_{k=1}^n
      p (\theta \cos (2 \pi q^k \omega)) d \omega.
    \]
  We recall from Lemma \ref{lem-estimate} that we can write $p(\theta \cos (2 \pi q^k \omega))$ in the form
\begin{equation} \label{coeff}
  b_0 (\theta) + b_1 (\theta) \cos (2 \pi q^k \omega) + b_2 (\theta) \cos (2 \pi 2 q^k \omega) + \dots + b_d (\theta)
  \cos( 2 \pi d q^k \omega),
\end{equation}  
where $b_0 = b_0 (\theta;d)$ is given by \eqref{eq:b_0_theta} and $b_j  = b_j(\theta;d) \geq 0$ for $j = 1, \ldots, d$.
Since $d >q$  by assumption, the $q$-th term in the Taylor expansion for $p(\theta \cos (2 \pi q^k x))$ is $(\theta \cos (2 \pi q^k x))^q / q!$. From \eqref{prod_cos_1} and \eqref{prod_cos_2} we see that the expansion of $(\cos y)^q$ into a linear combination of cosine functions contains the term $2^{-q+1} \cos (qy)$. We emphasize again that all coefficients, in the Taylor expansion of $e^y$ as well as in \eqref{prod_cos_1} and \eqref{prod_cos_2}, are non-negative. Thus the coefficient $b_q(\theta)$ in \eqref{coeff} is at least as large as the contribution coming from $(\theta \cos (2 \pi q^k x))^q / q!$, and so we have 
\begin{eqnarray}
  \label{bq-est}
b_q = b_q (\theta) & \geq & \frac{\theta^q}{q!} \frac{1}{2^{q-1}}.
\end{eqnarray}
By a similar reasoning the coefficient $b_1(\theta)$ in \eqref{coeff} is at least as large as the contribution coming from the linear term in the Taylor expansion, which is simply  $\theta \cos (2 \pi q^k x)$. Thus we have $b_1 (\theta) \geq \theta$. Once again using the fact that all coefficients are non-negative, in \eqref{coeff} as well as in \eqref{standard_id}, \eqref{prod_cos_1} and \eqref{prod_cos_2}, we have
$$
\int_0^1 \prod_{k=1}^n p(\theta \cos (2 \pi q^k \omega))~d\omega \geq \int_0^1 \prod_{k=1}^n \left(b_0 + b_1 \cos (2 \pi q^k \omega) + b_q \cos (2 \pi q^{k+1} \omega) \right) d\omega.
$$
Now the point is that there will always be interference between the term $b_q \cos (2 \pi q^{k+1} x)$ coming from index $k$, and the term $b_1 \cos (2 \pi q^{k+1} x)$ coming from index $k+1$. Let us assume that $n$ is even. Always combining two consecutive factors together, we have
\begin{eqnarray*}
& & \int_0^1 \prod_{k=1}^n \left(b_0 + b_1 \cos (2 \pi q^k \omega) + b_q \cos (2 \pi q^{k+1} \omega) \right) d\omega \\
& = & \int_0^1 \prod_{\ell=1}^{n/2} \left(b_0 + b_1 \cos (2 \pi q^{2\ell-1} \omega) + b_q \cos (2 \pi q^{2 \ell} \omega) \right) \left(b_0 + b_1 \cos (2 \pi q^{2 \ell} \omega) + b_q \cos (2 \pi q^{2 \ell+1} \omega) \right) d\omega \\
& \geq & \int_0^1 \prod_{\ell=1}^{n/2} \left( b_0^2 + b_q \cos (2 \pi q^{2\ell} \omega) b_1 \cos (2 \pi q^{2\ell} \omega) \right) d\omega \\
& \geq & \prod_{\ell=1}^{n/2} \left( b_0^2 + \frac{b_1 b_q}{2} \right) \\
& \geq & \prod_{\ell=1}^{n/2} \left(b_0^2 + \frac{\theta^{q+1}}{q! 2^q} \right).
\end{eqnarray*}
where the last inequality uses \eqref{bq-est} and $b_1(\theta) > \theta$. 
Consequently, we have
$$
\Lambda_q(\theta) =  \lim_{n \rightarrow  \infty} \frac{1}{n}  \log\E[ e^{\theta S_n}]  \geq \log \left( \underbrace{\left(1 + \frac{e^\theta \theta^{d+1}}{(d+1)!}\right)^{-1}}_{\to 1 \text{ as } d \to \infty} \sqrt{b_0^2 + \frac{\theta^{q+1}}{q! 2^q}} \right).
$$
Recall that $b_0$ depends on $d$ and $\theta$, and that we have $\log b_0 \to \widetilde{\Lambda} (\theta)$ as $d \to \infty$.
For every fixed  $\theta > 0$  since the logarithm
is a strictly  increasing  function, the term $\log \left( \sqrt{b_0^2 + \frac{\theta^{q+1}}{q! 2^q}} \right)$ converges to a quantity  that is  is strictly larger than $\widetilde{\Lambda} (\theta)$ as $d \to \infty$. Consequently, we have
$$
\Lambda_q(\theta) > \widetilde{\Lambda}(\theta),  \quad \mbox{ for all  }   \theta > 0. 
$$
From the properties of the Bessel function $B_0(\theta)$ it is easily seen that for $x \in (0,1)$ the supremum in the definition of $\widetilde I(x)$ is actually a maximum, and is attained at some (finite) value $\theta_x>0$. Consequently, we have
$$
I_q (x) = \sup_{\theta > 0} \left[\theta x  - \Lambda_q (\theta)\right]  =  \theta_x x  - \Lambda_q (\theta_x) < \theta_x x  - \widetilde{\Lambda} (\theta_x) = \widetilde{I} (x). 
$$
Thus, we have $I_q(x) < \widetilde{I}(x)$ for all $x \in (0,1)$.\\

In conclusion, we note that we can make the difference between $\Lambda_q$ and $\widetilde{\Lambda}$ quantitative. Recall that $\theta>0$ by assumption. Since $b_0(\theta)$ is a partial sum of $B_0(\theta)$, we have $b_0(\theta) \leq B_0(\theta)$. Furthermore, from the series expansion for $B_0(\theta)$ it is easily seen that $B_0(\theta) \leq e^\theta$. Thus $b_0^2 (\theta) \leq e^{2 \theta}$, and $b_0^2 + \frac{\theta^{q+1}}{q! 2^ q} \geq b_0^2 \left(1 +\frac{\theta^{q+1}}{q! 2^q e^{2 \theta}} \right)$. Thus, letting $d \to \infty$, we deduce that 
$$
\Lambda_q (\theta) - \widetilde{\Lambda} (\theta) \geq \frac{1}{2} \log \left(1 +\frac{\theta^{q+1}}{q! 2^q e^{2 \theta}} \right).
$$

\vskip 2mm
\noindent 
{\em Proof of (ii)}:  It follows from Proposition \ref{prop-mgf-PF} that for $\theta \in \mathbb{R}$, 
 $\Lambda_q(\theta) = \log \lambda_{\theta,q}$, where $\lambda_{\theta,q}>0$ is the largest eigenvalue of the Perron-Frobenius transfer operator
  defined in \eqref{PerFrobPerturb0}.
  Fixing $\theta \in \R$ and sending $q\to\infty$, the Riemann sums converge on the right-hand side of the
  definition in \eqref{PerFrobPerturb0} converge to the corresponding Riemann integrals; hence this sequence of operators converges in the norm topology to the operator
$$
(\widetilde\Phi_{\theta} g) (\omega)
=
\int_0^1 e^{\theta X_1(z)} g(z) dz
=
\widetilde \lambda_\theta \frac{\int_0^1 e^{\theta X_1(z)} g(z) dz}{\int_0^1 e^{\theta X_1(z)} dz} \cdot \mathbf{1},
$$
where $\widetilde \lambda_\theta := \int_0^1 e^{\theta X_1(z)} dz = e^{\widetilde \Lambda(\theta)}$, where
$\widetilde \Lambda$ is defined as in \eqref{def-tildeLambda}. Thus, $\widetilde\Phi_{\theta}/\widetilde \lambda_\theta$ is  a projection onto the line spanned by the function $\mathbf{1}$. The Perron-Frobenius eigenvalue of $\widetilde \Phi_{\theta, q}$ is $\widetilde \lambda_\theta$. Now, if $\theta\in\R$ stays constant and $q\to\infty$, we can view $\Phi_{\theta,q}$ as a perturbation of $\widetilde \Phi_\theta$.
By perturbation theory (see, e.g., \cite{Kato_book}),
we have the convergence of the Perron-Frobenius eigenvalues, that is,  $\lim_{q\to\infty}\lambda_{\theta,q} = \widetilde \lambda_\theta$ for every $\theta\in\R$. Taking the logarithm, we get $\lim_{q\to\infty} \Lambda_q(\theta) = \widetilde \Lambda(\theta)$.
Since the involved functions are convex, the convergence is, in fact, uniform on compact intervals. By taking the Legendre-Fenchel transform, it follows that $\lim_{q\to\infty} I_q(x) =  \widetilde I(x)$ locally uniformly on $(-1,1)$. \\

\noindent 
{\em Proof of (iii)}.
Lemma~\ref{lem:arcsine_moments} shows that for every $n \in \N$, whenever $m \in \N$ satisfies $m\leq q$, one has 
\[
\E[S_n^m] = \E[\widetilde{S}_n^m] = \frac{B_{m}(n)}{2^m}, 
\] 
or, in other words,  the moments of $S_n$ and $\widetilde{S}_n$ coincide for all $m\leq q$. 
  Since cumulants of order less than or equal to $q$ can be expressed in terms of
 moments of order less than or equal to $q$, we infer that as long as $m\leq q$, 
\[
 \kappa_m(S_n) = \kappa_m(\widetilde{S}_n) := n \cdot \kappa_m(\widetilde{X}_1), 
 \]
where $\kappa_m (Y)$  denotes the $m$th cumulant of a real-valued random variable $Y$.  
Hence, for $m\leq q$ and every $n\in\N$, we have 
\[
\left( \frac{d}{d\theta} \right)^m\Lambda_{q,n}(\theta)\Bigg|_{\theta=0} = \frac{1}{n}\kappa_m(S_n) =  \kappa_m(\widetilde{X}_1) =  \frac{1}{n} \kappa_m(\widetilde{S}_n) = \left( \frac{d}{d\theta} \right)^m\widetilde{\Lambda}(\theta)\Bigg|_{\theta=0} \,.
\] 
Now because the uniform convergence of the analytic functions $\Lambda_{q,n}\to\Lambda_q$ (established in Proposition \ref{prop-mgf-PF}) implies the convergence of the derivatives, we obtain (iii). \\

\noindent 
{\em Proof of (iv)}: In the case when $m=q+1$, a slight modification of the argument used to prove Lemma~\ref{lem:A_m_n_equals_B_m_n} shows that any solution to~\eqref{eq:syst_0} either satisfies $H_1= \ldots = H_n = 0$, or is a permutation of one of the solutions $q^k + \ldots +q^k - q^{k+1} = 0$ or $q^{k+1}-q^k -\ldots - q^k = 0$, where $k\in \{1,\ldots, n-1\}$. The total number of such exceptional solutions is $2 (q+1)(n-1)$, hence
$$
A_{q+1}(n) = B_{q+1}(n) + 2(q+1)(n-1).
$$
From Lemma~\ref{lem:A_m_n_moment} and Lemma~\ref{lem:arcsine_moments} it follows that
$$
\E[S_n^{q+1}]
=
\frac{A_{q+1}(n)}{2^{q+1}}
=
\frac{B_{q+1}(n) + 2(q+1)(n-1)}{2^{q+1}}
=
\E[\widetilde{S}_n^{q+1}]  + \frac{(q+1)(n-1)}{2^{q}}.
$$
The cumulant $\kappa_{q+1}(S_n)$ can be expressed as $\E [S_n^{q+1}]$ plus some polynomial function of the lower moments $\E [S_n^m]$ with $m\leq q$. A similar representation holds for the cumulant $\kappa_{q+1}(\widetilde S_n)$, and the moments of all orders $m\leq q$ of $S_n$ coincide with those of $\widetilde S_n$ by part~(ii) of Theorem~\ref{th2}. It follows that
$$
\kappa_{q+1}(S_n) = \kappa_{q+1}(\widetilde{S}_n)  + \frac{(q+1)(n-1)}{2^{q}}.
$$
For the derivatives of order $q+1$ of $\Lambda_{q,n}$ and $\widetilde \Lambda$ at $\theta=0$ we therefore obtain
$$
\Lambda_{q,n}^{(q+1)}(0)
=
\frac{1}{n}\kappa_{q+1}(S_n)
=
\frac{1}{n} \kappa_{q+1}(\widetilde{S}_n)  + \frac{(q+1)(n-1)}{2^{q}n}
=
\widetilde{\Lambda}^{(q+1)}(0)  + \frac{(q+1)(n-1)}{2^{q}n}.
$$
Letting $n\to\infty$ and using that the uniform convergence of analytic functions $\Lambda_{q,n} \to \Lambda_q$ implies convergence of their derivatives, we arrive at
$$
\Lambda_q^{(q+1)}(0)
=
\widetilde{\Lambda}^{(q+1)}(0)  + \frac{q+1}{2^{q}}.
$$
This proves (iv).

\subsubsection{Proof of Lemma~\ref{lem:I_q_+1}}
\label{subs-pf-lemIq}

We now present the proof of Lemma~\ref{lem:I_q_+1}.  
The idea is that in the lacunary sum $S_n$ all cosine functions $\cos (2\pi q^k U)$ are equal to $1$ at $U = 0$. Thus, $S_n$ is close to $n$ if the uniform random variable $U \sim {\rm Unif}(0,1)$ takes a value that is sufficiently close to $0$. To make this precise, fix $\eps \in (0,1)$. We have $\cos x \geq 1- x^2/2$. It follows that
$$
X_k = \cos (2 \pi q^k U) \geq 1-\eps
\quad
\text{provided that}
\quad
U \leq \frac{\sqrt{2\eps}}{2\pi} q^{-k}.
$$
Hence, if $U \leq \frac{\sqrt{2\eps}}{2\pi} q^{-n}$, then we have $S_n \geq  (1-\eps)n$. It follows that
\begin{align*}
I_q(1-\eps)  = -\lim_{n\to\infty} \frac 1n \log \mathbb P(S_n \geq  (1-\eps)n)
&\leq
-\lim_{n\to\infty} \frac 1n \log \mathbb P\Big(U\leq \frac{\sqrt{2\varepsilon}}{2\pi}q^{-n}\Big)
\cr &
\leq
-\lim_{n\to\infty} \frac 1n \log \left(\frac{\sqrt{2\eps}}{2\pi} q^{-n}\right)
\cr &  =
 \log q.
\end{align*}
Since 
this holds for every $\eps \in (0,1)$, by the lower semicontinuity of $I_q$, it follows that 
\[
I_q(+1) \leq \liminf_{\varepsilon\to 0}I_q(1-\varepsilon) \leq \log q.
\]
This completes the proof. 

\subsection{Proof of Theorem \ref{th3}}
\label{subs-pf-theoremC}

We know from Theorem~\ref{th2} and Proposition~\ref{prop:I_2_I_3} that there exists some sufficiently small
$\bar{x}_0>0$ such that $0<I_2(x_0) < I_3(x_0)$ for every $x_0\in\R$ with $0< x_0\leq \bar{x}_0$. By interleaving the powers of $2$ and $3$ appropriately, we shall construct an Hadamard gap sequence $(a_k)_{k\in\N}$ such that for all
$x_0 \in (0,\bar{x}_0)$, the corresponding partial sums $(S_n)_{n \in \N}$ satisfy 
\[
0< \liminf_{n\to\infty} -\frac{1}{n}\log \Pro(S_n \geq n x_0) < \limsup_{n\to\infty} - \frac{1}{n}\log \Pro(S_n \geq n x_0) <\infty.
\]
Since both $I_2$ and $I_3$ are continuous, there exist $\varepsilon_0\in(0,x_0)$ and $\delta_0>0$ such that
\[
\sup_{|x-x_0|\leq \varepsilon_0} I_2(x) + \delta_0 < \inf_{|x-x_0|\leq \varepsilon_0}I_3(x) \qquad\text{and}\qquad \inf_{|x-x_0|\leq \varepsilon_0}I_2(x)>\delta_0.
\]
Our construction proceeds inductively. Assume that for some $n\in\N$ we have constructed increasing positive integers $a_1,\dots,a_n$ such that
\[
-\frac{1}{n}\log \Pro(S_n \geq n x_0) > \inf_{|x-x_0|\leq \varepsilon_0} I_3(x) -\frac{\delta_0}{3} =: c_+ .
\]
We want to extend the sequence $a_1,\dots,a_n$ to a longer sequence $a_1,\dots,a_N$, with $N\in\N$, $N>n$ in such a way that
\[
-\frac{1}{N}\log \Pro(S_N \geq N x_0) < \sup_{|x-x_0|\leq \varepsilon_0} I_2(x) + \frac{\delta_0}{3} =: c_-.
\]
Note that $0<c_-<c_+$. This is done as follows. We define $a_{n+1} := 2^m$, where $m\in\N$ is any number such that $2^m>2a_n$ (to guarantee the Hadamard gap condition) and $m>n+1$. Further, we define $a_{n+\ell}:=2^{m+(\ell-1)}$ so that with $N=n+(N-n)$, we have $a_{N}:= 2^{N-n+m-1}$. We choose $N\in\N$ sufficiently large, in particular such that $2m/N < \varepsilon_0/5$. Clearly,
\[
\left|\sum_{k=1}^n \cos(2\pi a_k x_0) \right| \leq n
\qquad
\text{and}
\qquad
\left|\sum_{k=1}^{m-1} \cos(2\pi 2^k x_0) \right| \leq m.
\]
Therefore, by replacing the first $n$ elements $a_1,\dots,a_n$ by $m-1$ powers of $2$, more precisely, by $2,2^2,\dots, 2^{m-1}$, respectively, and using the specific choice of $a_{n+1},a_{n+2},\dots$ together with the two estimates in the previous display (which guarantee that the replacement of $n$ by $m-1$ cosine terms yields an error bounded above by $n+m$), we obtain
\begin{align*}
-\frac{1}{N}\log \Pro(S_N \geq N x_0)
&
\leq
-\frac{1}{N}\log \Pro\bigg(\sum_{k=1}^{N-n+m-1}\cos(2\pi 2^k x_0) \geq N x_0 + (n+m)\bigg)\\
&
\leq
-\frac{1}{N}\log \Pro\bigg(\sum_{k=1}^{N}\cos(2\pi 2^k x_0) \geq N x_0 + (n+m) + (m-n-1)\bigg) \\
&\leq
-\frac{1}{N}\log \Pro\bigg(\sum_{k=1}^{N}\cos(2\pi 2^k x_0) \geq N x_0 + 2m\bigg)\\
& \leq -\frac{1}{N}\log \Pro\bigg(\sum_{k=1}^N\cos(2\pi 2^k x_0) \geq N (x_0 +\varepsilon_0/5)\bigg),
\end{align*}
where we used that $2m/N < \varepsilon_0/5$. The latter expression converges, as $N\to\infty$, to $I_2(x_0+\varepsilon_0/5)$. Hence, making $N\in\N$ larger, if necessary,  we obtain
\[
-\frac{1}{N}\log \Pro(S_N \geq N x_0)   < I_2(x_0+\varepsilon_0/5) +\frac{\delta_0}{3} \leq \sup_{|x-x_0|\leq \varepsilon_0} I_2(x) + \frac{\delta_0}{3} = c_{-}.
\]
Now we can continue this argument back and forth, by adding strings of consecutive powers of $2$ in odd steps and strings of  powers of $3$ in even steps, we can construct an infinite sequence $(a_k)_{k\in\N}$ for which $-\frac{1}{n}\log \Pro(S_n \geq n x_0)$ is infinitely often smaller than $c_-$ and infinitely often larger than $c_+$.

\subsection{Proof of Theorem \ref{th4}}\label{subs-pf-TheoremD}
Recall that the  i.i.d.~sequence  $Y = (Y_k)_{k \in\N}$ is defined on a common probability space
  $(\yspace, \yfilt,\py)$ 
    with each $Y_k$  uniformly distributed on the discrete set
  \begin{equation}
    \label{def-discset_recalled}
  \discset_k := \left\{h 2^{\lceil k^{2/3} \rceil}: h \in \mathbb{Z}, 0 \leq h \leq 2^{\lceil k^{2/3} \rceil} \right\}, \quad   k \in \N. 
  \end{equation} 
Since by definition $a_k^Y = 2^k + Y_k$, $(a_k^Y)_{k\in \N}$ is also a sequence
of independent random variables defined on  $(\yspace, \yfilt, \py)$. 
  We also assume
  (without loss of generality) that the independent uniform random variable
$U$ is realized as the identity map on the space
$([0,1], \mathcal{B}(0,1), \lambda)$ and, since $U$ and $Y$ are independent, that both
  $Y$ and $U$ are defined on the product measure space $(\yspace \times [0,1], \yfilt \otimes \mathcal{B}(0,1), \py \otimes \lambda)$. Throughout the argument, fix $\theta \in \R$.  The proof proceeds in several steps. \\

\noindent 
{\em Step 1.  Construct a suitable partition of the integers.}   

For any large  $n$,  we split the set of all positive integers into disjoint sets $\Delta_1, \Delta_2, \dots$ and $\Delta_1', \Delta_2', \dots$, which are defined via the
  following recursive construction. First, set $\Delta_1 := \{1, \dots, n^{1/2}\}$, where for notational simplicity,
we assume that $n^{1/2}$ is an integer.  Let the set $\Delta_1'$ contain the next $n^{2/5}$ smallest positive integers not already contained in $\Delta_1$, where (again for notational simplicity) we assume that $n^{2/5}$  is also an integer. Then, for each $i\in\N$, we recursively define $\Delta_{i+1}$ to  contain the $n^{1/2}$ smallest positive integers not already contained in $\bigcup_{j=1}^{i} (\Delta_j \cup \Delta'_j)$, and the set   $\Delta'_{i+1}$ to contain
  the $n^{2/5}$ smallest positive integers not already contained in $\Delta_{i+1} \cup \bigcup_{j=1}^{i} (\Delta_j \cup \Delta'_j)$.
This decomposition can be characterized by the following requirements: 
\begin{itemize}
 \item $\Delta_i < \Delta_i' < \Delta_{i+1}$ for all $i\in\N$, where the inequality is understood to hold element-wise
 \item $\bigcup_{i=1}^\infty \left( \Delta_i \cup \Delta_{i}' \right) = \mathbb{N}$.
 \item $\# \Delta_i = n^{1/2}$ for all $i\in\N$, and $\# \Delta_i' = n^{2/5}$ for all $i\in\N$.
\end{itemize}
The philosophy is that the primed index sets are sufficiently large to cause a strong ``independence'' between the trigonometric functions in the non-primed sets, while at the same time the total cardinality of the primed index sets is so small that they are asymptotically negligible. The precise choice of $n^{1/2}$ and $n^{2/5}$ for the cardinalities of the $\Delta$ and $\Delta'$ blocks is somewhat arbitrary, the relevant facts are that the one type of block is significantly larger than the other, and that both types of blocks are not too small in comparison with $n$.\\
For $i \in \N$, let $\delta_{i,\min}$ and $\delta_{i,\max}$ denote the smallest and largest integers in $\Delta_i$, respectively. Then our construction ensures that 
\begin{equation} \label{sep_0} 
  \delta_{i,\max} + n^{2/5} <  \delta_{i+1,\min},  ~\forall i \in \N, \quad \mbox{ and }  \quad \max_{1 \leq i \leq M_n} \delta_{i,\max} + n^{2/5} \leq n,   
\end{equation}
where 
\begin{equation}
  \label{def-integn}
  \integ_n := \min \left\{ \integ \in \mathbb{N}:  \{1, \dots, n\} \subset \bigcup_{i=1}^{\integ+1} \left( \Delta_i \cup \Delta_{i}' \right)\right\}
  \leq \sqrt{n}, 
\end{equation}
with the last inequality being a simple consequence of the fact that $|\Delta_i| = \sqrt{n}$ for each $i \in \N$. \\

\noindent 
{\em Step 2. Bound  the moment generating function in terms of polynomial integrals. } 

Recall that $\discsetinf = \otimes_{k\in\N} \discset_k$, where the definition of the discrete set 
  ${\mathcal D}_k$ was repeated again in \eqref{def-discset_recalled}, and for $y \in\discsetinf$, 
   $a_k^y = 2^k + y_k$. Recall also that $S_n^y(\omega) = \sum_{k=1}^n \cos (2 \pi a_k^y \omega)$, $\omega\in[0,1]$.

\begin{lemma}
  \label{lem-thus1}
  Fix $n\in\N$ sufficiently large such that $n^{2/5} \leq \integ_n$.   Then, for $y \in \discsetinf$ and
  $\omega \in [0,1]$,
    \begin{equation} \label{thus_1}
      e^{-5 \theta n^{9/10}} \int_0^1 \term^y_n(\omega) d\omega \leq \int_0^1 e^{\theta S_n^y(\omega)} d\omega 
      \leq e^{5 \theta n^{9/10}} \int_0^1 \term^y_n(\omega) d\omega,
    \end{equation}
    where for $\omega \in [0,1]$
    \begin{equation}
      \label{poly}
      \term_n^y(\omega) := \left(\prod_{n^{2/5} \leq i \leq \integ_n} ~\prod_{k \in \Delta_i} e^{\theta \cos (2 \pi a_k^y \omega)} \right).
    \end{equation}
    Consequently, for any $\ve > 0$, there exists $d = d(\ve) \in \N$ such that
    the Taylor polynomial $p = p_{d(\ve)}$ of length $d(\ve)$
    defined in \eqref{def-poly} satisfies 
\begin{equation} \label{thus_2}
(1-\ve)^n \int_0^1 \term^y_n(\omega) d\omega \leq \displaystyle \int_0^1 \prod_{i=n^{2/5}}^{\integ_n} \prod_{k \in \Delta_i} p(\theta \cos (2 \pi a_k^y \omega)) ~d\omega \leq (1+\ve)^n \int_0^1 \term^y_n(\omega) d\omega,
\end{equation}
for every $y \in \discsetinf$ and for all sufficiently large $n\in\N$.
 \end{lemma}
  \begin{proof}
    Fix  $n\in\N$ as in the statement of the lemma. Also, fix
      $y \in \discsetinf$ and for notational conciseness, omit all dependence on $y$. Then
     for  any  $\omega \in [0,1]$, we can split 
\begin{eqnarray*}
  e^{\theta S_n(\omega)}  &=&  \prod_{k=1}^n e^{\theta \cos (2 \pi a_k \omega)} \\
& = & \underbrace{\left(\prod_{1 \leq i < n^{2/5}} ~\prod_{k \in \Delta_i \cup \Delta_i'} e^{\theta \cos (2 \pi a_k \omega)} \right)}_{=:\term^{(1)}(\omega)} \times \underbrace{\left(\prod_{n^{2/5} \leq i \leq \integ_n} ~\prod_{k \in \Delta_i} e^{\theta \cos (2 \pi a_k \omega)} \right)}_{=:\term^{(2)}(\omega)} \times \\
& & \qquad \times \underbrace{\left(\prod_{n^{2/5} \leq i \leq \integ_n} ~\prod_{k \in \Delta_i'} e^{\theta \cos (2 \pi a_k \omega)} \right)}_{=:\term^{(3)}(\omega)} \times \underbrace{\left( \prod_{\substack{1 \leq k \leq n,\\ k \not\in \bigcup_{i=1}^{\integ_n} (\Delta_i \cup \Delta_i')}} e^{\theta \cos (2 \pi a_k \omega)} \right)}_{=:\term^{(4)}(\omega)}.
\end{eqnarray*}
We will show below that  $\term^{(1)}, \term^{(3)}$, and $\term^{(4)}$ are
all sub-exponential in $n$ (that is, their logarithms are all sublinear in $n$),
and thus  these three factors will be negligible in comparison with $\term^{(2)}$,
whose  logarithm  is linear in $n$.  Indeed, first note that 
by construction $\term^{(1)}$ is a product of at most $2 n^{2/5} n^{1/2}$ factors, each of which is trivially between $e^{-\theta}$ and $e^{\theta}$, so in total we have $e^{-2\theta n^{9/10}} \leq \term^{(1)}(\omega) \leq e^{2\theta n^{9/10}}$ for all $\omega \in[0,1]$. Next, the product $\term^{(3)}$ contains all contributions coming from the complete short ``primed'' blocks $\Delta'$; the purpose of these blocks was just to separate the longer blocks, and $\term^{(3)}$  is also small in comparison with
$\term^{(2)}$. Since the product $\term^{(3)}$ has a total of at most $\integ_n n^{2/5} \leq n^{9/10}$ many factors,  we have $e^{-\theta n^{9/10}} \leq \term^{(3)}(\omega) \leq e^{\theta n^{9/10}}$ for all $\omega \in[0,1]$.  Lastly, the product $\term^{(4)}$ is split off since it does not cover a full block; this is no problem, since $\term^{(4)}$ only has a small number of factors. More precisely, since
by \eqref{def-integn}, $n - M_n \leq M_{n+1} - M_n \leq 2 \sqrt{n}$, 
we have $e^{-2\theta n^{1/2}} \leq \term^{(4)}(\omega) \leq e^{2\theta n^{1/2}}$ for all $\omega \in[0,1]$. 
Overall, this implies  $e^{-5 \theta n^{9/10}} \leq \term^{(1)}(\omega) \term^{(3)}(\omega) \term^{(4)}(\omega) \leq e^{5 \theta n^{9/10}}$ for all $\omega \in[0,1]$.
 When combined with  the last display, and the observation that everything inside the integrals is positive, this 
yields
\eqref{thus_1} with $\term_n := \term^{(2)}$,  which agrees
with the expression in \eqref{poly}. 
The second estimate \eqref{thus_2} is then
a simple consequence of \eqref{thus_1}, \eqref{eps_approx} of Lemma \ref{lem-taylor} 
and the relations 
$|\Delta_i| = \sqrt{n}$ for all $i$ and   $\integ_n \leq \sqrt{n}$.  \end{proof}

  \noindent
      {\em Step 3. Evaluate the integral  $\int_0^1 \prod_{i=n^{2/5}}^{\integ_n} \prod_{k \in \Delta_i} p(\theta \cos (2 \pi a_k^y \omega)) ~d\omega$ from \eqref{thus_2}.}
      The key idea  is to first show that we can take the product $\prod_{i=n^{2/5}}^{\integ_n}$ outside the integral;
see \eqref{indep_prod} below.  In other words, we show that there are no correlations between cosine functions with indices from different blocks $\Delta_i$ and $\Delta_j$, for $i, j, i \neq j$, in the range, and thus, that 
      it  is possible to evaluate all integrals entirely within each block. 
 Indeed, this was the purpose of the construction of $\Delta_i$ and $\Delta'_i$ in Step 1.  Then we simplify each of the 
 integrals in the product using the expansion for the polynomial
 $p$ obtained in Lemma \ref{lem-estimate}. 
   Indeed, recall from \eqref{exp_pk} and \eqref{eqn:b_0_theta} of 
   that lemma  that for $d \in \N$,   there exist nonnegative 
   coefficients $b_j = b_j(\theta;d), j = 0, \ldots, d$, such that for
      all $k \in \N$, the Taylor polynomial $p = p_d$ satisfies 
\begin{equation} \label{b_form}
  p \left(\theta \cos (2 \pi a_k^y \omega) \right) =
  \sum_{j=0}^d b_j (\theta)  \cos (2 \pi j a_k^y \omega), 
\end{equation}
where the zeroth coefficient $b_0 = b_0(\theta; d)$ is given explicitly by the finite series   in \eqref{eqn:b_0_theta}. 
To shorten notations we suppress the dependence of $b_0,~b_1,~ \dots$ on $\theta$  and $d$ in the formulas below.\\

      \begin{prop}
        \label{prop-interchange}
        Fix $d \in \N$ and $p = p_d$ the Taylor polynomial of length $d$. 
        Then, for all sufficiently large $n$, and every  $y \in \discsetinf$,
\begin{equation} \label{indep_prod}
  \int_0^1 \prod_{i=n^{2/5}}^{\integ_n} \prod_{k \in \Delta_i} p(\theta \cos (2 \pi a_k^y \omega)) ~d\omega = \prod_{i=n^{2/5}}^{\integ_n} \int_0^1 \prod_{k \in \Delta_i} p(\theta \cos (2 \pi a_k^y \omega)) d\omega. 
\end{equation}
Furthermore, for $i \in \{ n^{2/5}, \ldots, \integ_n\}$,   
\begin{align} 
  \label{form_prod_int0}
  c_0^{(i)}(y) := & \int_0^1 \prod_{k \in \Delta_i} p(\theta \cos (2 \pi a_k^y \omega)) d\omega   \\
= & b_0^{\sqrt{n}} + \sum \sum \sum b_0^{\sqrt{n}- \ell} b_{j_1} \cdot \ldots \cdot b_{j_\ell} \sum  \frac{\ell!}{2^{\ell-1}}  \mathbbm 1_{\left\{j_1 a_{k_1}^y + s_2 j_2 a_{k_2}^y + \dots + s_\ell j_\ell a^y_{k_\ell} = 0 \right\}},  \label{form_prod_int} 
\end{align}
where  the
four summations in the displayed formula above are taken over the ranges (in the order of appearance) 
\begin{equation} \label{sum_r}
\sum_{1 \leq \ell \leq \sqrt{n}}, \qquad \sum_{\substack{(k_1, \dots, k_\ell) \in \Delta_i,\\k_1 > \dots > k_\ell}}, \qquad \sum_{\substack{(j_1, \dots, j_\ell) \in \{0, \dots, d\}^\ell,\\(j_1, \dots, j_\ell) \neq (0,\dots,0)}}, \qquad \sum_{(s_2, \dots, s_\ell) \in \{-1,1\}^{\ell-1} },
\end{equation}
and the coefficients $b_j = b_j(\theta; d),$ $j=0, \ldots d$, are as in \eqref{b_form}. 
Furthermore,  for all sufficiently large $i$,  given  $\ell$ and $k_m, j_m, s_m, m = 1, \ldots, \ell$ as in \eqref{sum_r}, we have
\begin{equation}
  \label{zero}
  j_1 a_{k_1}^y + s_2 j_2 a_{k_2}^y + \dots + s_\ell j_\ell a^y_{k_\ell} = 0  \quad \Rightarrow \quad 
  j_1 y_{k_1} + s_2  j_2 y_{k_2} +  \dots +   s_\ell j_\ell y_{k_\ell} = 0. 
\end{equation}
      \end{prop}
           
     \begin{proof}
  Fix $y = (y_k)_{k \in \N} \in \discsetinf$.  We will start by
  establishing \eqref{form_prod_int} and \eqref{zero}. 
  Multiplying out the product $\prod_{k \in \Delta_i} p(\theta \cos (2 \pi a_k^y \omega))$ within a certain fixed block $\Delta_i$,  using \eqref{b_form} and the cosine product trigonometric
  identity \eqref{standard_id} we obtain 
\begin{align} 
&  \prod_{k \in \Delta_i} p(\theta \cos (2 \pi a_k^y \omega)) \label{form_prod}\\
=  & b_0^{\sqrt{n}} + \sum \sum \sum b_0^{\sqrt{n}- \ell} b_{j_1} \cdot \ldots \cdot b_{j_\ell} \sum   \frac{\ell!}{2^{\ell-1}} \cos \big( 2 \pi (j_1 a_{k_1}^y + s_2 j_2 a_{k_2}^y + \dots + s_\ell j_\ell a_{k_\ell}^y )\omega \big), \nonumber
\end{align}
where the four summations in the displayed formula above are taken over the ranges (in the order of appearance) in \eqref{sum_r},
and the  power $\sqrt{n}$ in the constant term $b_0^{\sqrt{n}}$ and the coefficient
  $b_0^{\sqrt{n}- \ell}$ arises from the fact that 
 $|\Delta_i| = \sqrt{n}$. 
    Note that \eqref{form_prod}  shows that
$\prod_{k \in \Delta_i} p(\theta \cos (2 \pi a_k^y \omega))$ can be written as the sum of the constant term $b_0^{\sqrt{n}}$ (which would be the contribution for the ``independent'' case;  see
\eqref{eq:int_prod_p} of Lemma \ref{lem-estimate}) 
plus a linear combination of cosine functions with frequencies
\begin{equation} \label{exp}
j_1 a_{k_1}^y \pm \dots \pm j_\ell a_{k_\ell}^y,    
\end{equation} 
the latter following from the trigonometric identity \eqref{standard_id}. 
Assume that the expression in \eqref{exp} is non-zero. Recall that $a_k^y = 2^k +y_k$, where $y_k$ takes values in  $\discset_k := \{h 2^{k^{2/3}}: ~0 \leq h \leq 2^{k^{2/3}}\}$; here and in the sequel we write $2^{k^{2/3}}$ for $2^{\lceil k^{2/3} \rceil}$ for  notational conciseness. Substituting $a_k^y = 2^k + y_k$ into \eqref{exp}, we can rewrite the frequency of the cosine function as
\begin{equation} \label{frequencies}
j_1 (2^{k_1} + y_{k_1}) \pm \dots \pm j_\ell (2^{k_\ell} + y_{k_\ell}) = \underbrace{j_1 2^{k_1} \pm \dots \pm j_\ell 2^{k_\ell}}_{\text{fixed part}} + \underbrace{j_1 y_{k_1} \pm \dots \pm j_\ell y_{k_\ell}}_{\text{$\discsetinf$-dependent part}}, 
\end{equation}
which is different from zero only if at least one of the parts is non-zero. 
Note that by \eqref{sum_r}, the absolute value of the fixed part in this expression, whenever it is non-zero, has a value between
$$
2^{\delta_{i,\min}} \qquad \text{and} 
\qquad 
d \sqrt{n} 2^{\delta_{i,\max}},
$$
where recall $\delta_{i,\min}$ and $\delta_{i,\max}$, respectively,
are the smallest and  largest elements of $\Delta_i$. Indeed, the upper bound is trivial and since $k_1>k_2>\dots>k_\ell$,
  we also obtain the lower bound: 
\[
|j_1 2^{k_1} \pm \dots \pm j_\ell 2^{k_\ell}| = |2^{k_\ell}|\cdot \underbrace{|j_12^{k_1-k_\ell}\pm\dots\pm j_\ell|}_{\geq 1,\,\,\text{ since non-zero}} \geq 2^{\delta_{i,\min}}.
\] 
Similarly, recalling  the structure  of ${\mathcal D}_k$ from
\eqref{def-discset_recalled},  
the $\discsetinf$-dependent part, whenever it is non-zero, has
absolute value between
$$
2^{\delta_{i,\min}^{2/3}} \qquad \text{and} \qquad 
d \sqrt{n} 2^{\delta_{i,\max}^{2/3}}.
$$
Thus (if both are non-zero), the absolute value of the sum of  the
  fixed and $\discsetinf$-dependent parts always lies between 
\[
\frac{1}{2}2^{\delta_{i,\min}} \qquad \text{and} 
\qquad 2d\sqrt{n}2^{\delta_{i,\max}}.
\]
Hence, the product $\prod_{k \in \Delta_i} p(\theta \cos (2 \pi a^y_k \omega))$ can be written as a constant term, plus a linear combination of cosine functions,  each of which has a frequency that  is contained in the range
\begin{equation} \label{range}
\bar{{\mathcal R}}_i := \left[ 2^{\delta_{i,\min}^{2/3}}, d \sqrt{n} 2^{\delta_{i,\max}^{2/3}} \right] \cup \left[\frac{1}{2}2^{\delta_{i,\min}}, 2 d \sqrt{n} 2^{\delta_{i,\max}} \right].
\end{equation}
In other words, from \eqref{form_prod} and the above discussion, it follows that 
\begin{equation} \label{sets_si}
\prod_{k \in \Delta_i} p(\theta \cos (2 \pi a_k^y \omega)) = c_0^{(i)}(y) + \sum_{m^{(i)} \in \mathcal{R}_i} c_{m^{(i)}}^{(i)}(y) \cos (2 \pi m^{(i)} x)
\end{equation}
for some appropriate set $\mathcal R_i \subset \bar{\mathcal R}_i $ of positive integers,  
and appropriate coefficients $c_m^{(i)}(y)$.
Note that $c_0^{(i)}$ may differ from $b_0^{\sqrt{n}}$ since some frequencies of the form \eqref{exp} may vanish, and the coefficients of the corresponding cosines would then contribute to the  constant term.
Using \eqref{form_prod} and the fact that the integral over $[0,1]$  of any cosine term in that
expansion with a non-zero frequency vanishes, we obtain 
\eqref{form_prod_int}. Note that 
the dependence of  $c_0^{(i)} = c_0^{(i)} (y)$  on
$y$  arises because the value of the indicator 
\[  
  {\mathbbm 1}_{\left\{j_1 a_{k_1}^y + s_2 j_2 a_{k_2}^y + \dots + s_\ell j_\ell a_{k_\ell}^y = 0 \right\}} 
\]
depends on $(y_{k_1}, \ldots, y_{k_\ell})$ via the values of $a_{k_1}^y, \dots, a_{k_\ell}^y$.

We now turn to the proof of \eqref{zero}. Recall that we  constructed our blocks $\Delta_i, \Delta'_i$ and defined $\integ_n$ such that
$\delta_{i+1,\min} \geq \delta_{i,\max} + n^{2/5}$ and for $i \leq \integ_n$, $\delta_{i,\max} + n^{2/5} \leq n$,  see \eqref{sep_0}, 
which together with the  mean-value theorem  implies  that 
\[ \delta_{i+1,\min}^{2/3} \geq \left(\delta_{i,\max} + n^{2/5}\right)^{2/3}
\geq  \delta_{i,\max}^{2/3}  + \frac{2}{3} \frac{1}{(\delta_{i,\max} + n^{2/5})^{1/3}} n^{4/15}
  \geq \delta_{i,\max}^{2/3}  + \frac{2}{3} n^{1/15}, 
  \]
  and hence, for $1 \leq i \leq \integ_n$, 
\begin{equation} \label{sep_1} 
2^{\delta_{i+1,\min}} \geq 2^{n^{2/5}+ \delta_{i,\max}} \quad \mbox{  and  } \quad 2^{\delta_{i+1,\min}^{2/3}} \geq 
2^{\delta_{i,\max}^{2/3} + \frac{2}{3} n^{1/15}}, 
\end{equation}
Also, for sufficiently large $n$ and  $n^{2/5} \leq i \leq \integ_n$, 
note that  $\delta_{i,\min} \geq (n^{2/5})(n^{1/2}) = n^{9/10} = n^{7/30 + 2/3}  \geq n^{7/30} + n^{2/3} \geq n^{7/30}+ \delta_{\integ_n,\max}^{2/3}$,
and so
  \begin{equation} \label{sep_3}
    \frac{1}{2} 2^{\delta_{n^{2/5},\min}} \geq \frac{2^{n^{7/30}-1}}{d \sqrt{n}} \left( d \sqrt{n} 2^{\delta^{2/3}_{M_n,\max}}\right).
      \end{equation}
The  last inequality shows that for all sufficiently large $n\in\N$,  any ``$\discsetinf$-dependent part'' of a
frequency that could originate from some product with indices in $\Delta_i$ (with $n^{2/5}\leq i \leq \integ_n$) is of a much smaller order than the smallest non-zero ``fixed'' part that we could encounter from such blocks, which proves \eqref{zero}.
(This is why we split off the product $\term^{(1)}$ with the frequencies in $\Delta_i \cup \Delta_i'$ for $1 \leq i < n^{2/5}$  earlier in Lemma \ref{lem-thus1},
since the frequencies there are so small that their fixed parts could cause correlations with the
$\discsetinf$-dependent parts coming from blocks with higher indices.)

To complete the proof of the proposition, it only remains to prove \eqref{indep_prod}.
To show how our construction facilitates control of the value of the integral 
\begin{equation} \label{contri}
\int_0^1 \prod_{i=n^{2/5}}^{\integ_n} \prod_{k \in \Delta_i} p(\theta \cos (2 \pi a_k^y \omega)) d\omega, 
\end{equation}
note that  \eqref{sets_si} implies that we have for $y \in \discsetinf$ and $\omega \in [0,1]$, 
\begin{equation} \label{r_h_s_of} 
  \prod_{i=n^{2/5}}^{\integ_n} \prod_{k \in \Delta_i} p(\theta \cos (2 \pi a_k^y \omega))
= \prod_{i=n^{2/5}}^{\integ_n} \left( c_0^{(i)}(y) + \sum_{m^{(i)} \in \mathcal{R}_i} c_{m^{(i)}}^{(i)}(y) \cos (2 \pi m^{(i)} \omega) \right), 
\end{equation}
with $c_0^{(i)}$ as in \eqref{form_prod_int} and $c_{m^{(i)}}^{(i)}$ other coefficients as described above (whose precise values will not matter for what follows). 
When multiplying out the terms in the product on the right-hand side of \eqref{r_h_s_of}, for each $i$ in the range $n^{2/5} \leq i \leq \integ_n$ we can either choose the factor $c_0^{(i)}(y)$ or a factor of the form $c_{m^{(i)}}^{(i)} (y)\cos (2 \pi m^{(i)} \omega)$ for some $m^{(i)} \in \mathcal{R}_i$. That is, we can write the right-hand side of \eqref{r_h_s_of} as
$$
\sum_{\mathcal{U}, \mathcal{V}} \left( \prod_{i \in \mathcal{U}} c_0^{(i)} (y)\right) \left( \prod_{i \in \mathcal{V}} \sum_{m^{(i)} \in \mathcal{R}_i} c_{m^{(i)}}^{(i)}(y) \cos (2 \pi m^{(i)} \omega) \right).
$$
where the sum is taken over all sets $\mathcal{U}, \mathcal{V}$ that form a disjoint partition  of
$\{n^{2/5}, \dots, \integ_n\}$, i.e.,\ $\mathcal{U} \cap \mathcal{V} = \emptyset$ and $\mathcal{U} \cup \mathcal{V} = \{n^{2/5}, \dots, \integ_n\}$. Assume that $\mathcal{V}$ is non-empty. Then using the standard trigonometric identity \eqref{standard_id} we can expand 
$$
\prod_{i \in \mathcal{V}} \sum_{m^{(i)} \in \mathcal{R}_i} c_{m^{(i)}}^{(i)}(y) \cos (2 \pi m^{(i)} \omega)
$$
into a linear combination of cosine-functions with frequencies of the form
$$
\sum_{i \in \mathcal{V}} \pm m^{(i)}, \qquad m^{(i)} \in \mathcal{R}_i.  
$$
Since  $\mathcal{R}_i$ is contained in the range set $\bar{\mathcal{R}}_i$ defined in \eqref{range}, and since we have the estimates \eqref{sep_1}
and \eqref{sep_3} separating these respective ranges for different values of $i$, it is not possible that the linear combination equals zero (provided that $n$ is large enough). Thus our construction ensures that all frequencies of cosine-functions in this linear combination are non-zero, which implies that their integrals vanish over $[0,1]$, so that we have
$$
\int_0^1 \prod_{i \in \mathcal{V}} \sum_{m^{(i)} \in \mathcal{S}_i} c_{m^{(i)}}^{(i)}(y) \cos (2 \pi m^{(i)} \omega) d\omega = 0, 
$$
and consequently, 
$$
\int_0^1 \left( \prod_{i \in \mathcal{U}} c_0^{(i)} (y)\right) \left( \prod_{i \in \mathcal{V}} \sum_{m^{(i)} \in \mathcal{R}_i} c_{m^{(i)}}^{(i)}(y) \cos (2 \pi m^{(i)} \omega) \right)d\omega = 0,
$$
whenever $\mathcal{V}$ is non-empty. Thus, the only term that actually contributes to the value of \eqref{contri} is when all indices $i$ are contained in $\mathcal{U}$ and  $\mathcal{V} = \emptyset$. The 
contribution of this case to the integral is
$$
\int_0^1 \prod_{i=n^{2/5}}^I c_0^{(i)}(y) d\omega = \prod_{i=n^{2/5}}^I c_0^{(i)}(y),
$$
so that in total we have for every $y \in \discsetinf$, 
$$
\int_0^1 \prod_{i=n^{2/5}}^I \prod_{k \in \Delta_i} p(\theta \cos (2 \pi a_k^y \omega)) d\omega = \prod_{i=n^{2/5}}^I c_0^{(i)}(y) = \prod_{i=n^{2/5}}^{\integ_n} \int_0^1 \prod_{k \in \Delta_i} p(\theta \cos (2 \pi a_k^y \omega)) d\omega.
$$
This is \eqref{indep_prod} and completes the proof of the proposition. 
\end{proof}

\noindent
{\em Step 4.  Give an explicit formula for $\prod_{i=n^{2/5}}^{\integ_n} \int_0^1 \prod_{k \in \Delta_i} p(\theta \cos (2 \pi a_k^Y \omega)) d\omega$ which holds with large $\py$-probability.} We will prove the following result.   

  \begin{lemma}
    \label{lem-estofexpect}
    Let   $Y = (Y_k)_{k \in \N}$ be the sequence  
of independent random variables, with each $Y_k$ uniformly distributed on the set ${\mathcal D}_k$ defined in \eqref{def-discset_recalled}. 
Then for every $d \in \N$, with  $p = p_d$, the Taylor  polynomial of length $d$, and $b_0 = b_0(\theta; d)$ as in \eqref{eqn:b_0_theta}, we have 
\begin{equation}\label{estofexpect}
 \py \left( \prod_{i=n^{2/5}}^{\integ_n} \int_0^1 \prod_{k \in \Delta_i} p\big(\theta \cos (2 \pi a_k^Y \omega)\big) d\omega = \prod_{i=n^{2/5}}^{M_n} b_0^{\sqrt{n}}\right) \geq 1 - n^{-3/2},
\end{equation}
for all sufficiently large $n\in\N$. 
  \end{lemma}  
  \begin{proof}
    Fix $d \in \N$ and set $p = p_{d}$ to  be the corresponding Taylor polynomial, and let $b_j := b_j(\theta; d),$
    $j = 0, 1, \ldots, d$, be
    the associated coefficients as presented in Equations \eqref{exp_pk} and \eqref{eqn:b_0_theta} of Lemma \ref{lem-estimate} (see also \eqref{b_form}). 
For any $y \in \discset^\infty$, let $c_0^{(i)}(y)$ be defined as in  \eqref{form_prod_int}. Combining   \eqref{form_prod_int} and \eqref{zero} in Proposition \ref{prop-interchange} with
the fact that $i \geq n^{2/5}$ in   \eqref{form_prod_int}, we see that for all sufficiently large $n\in\N$ and for all $i$ in the range $ n^{2/5}  \leq i \leq \integ_n$ we have
\begin{align} 
& \int_0^1 \prod_{k \in \Delta_i} p(\theta \cos (2 \pi a_k^y \omega)) d\omega \nonumber\\
= & b_0^{\sqrt{n}} + \sum \sum \sum b_0^{\sqrt{n}- \ell} b_{j_1} \cdot \ldots \cdot  b_{j_\ell} \sum \frac{\ell!}{2^{\ell-1}}
\mathbbm{1}_{\{j_1 y_{k_1} + s_2 j_2 y_{k_2} + \dots + s_\ell j_\ell y_{k_\ell} = 0 \}}, \label{form_prod_int_2a}
\end{align}
with the summation ranges as specified in \eqref{sum_r}.

We now  estimate  the probability of the event $\{j_1 Y_{k_1} + s_2 j_2 Y_{k_2} + \dots + s_\ell j_\ell Y_{k_\ell} = 0\}$. 
We recall that by assumption, $Y_{k_1}, \dots, Y_{k_\ell}$ are independent discrete random variables, and that $s_2, \dots, s_\ell$ are just some plus/minus signs. In principle the distribution of $j_1 Y_{k_1} + s_2 j_2 Y_{k_2} + \dots + s_\ell j_\ell Y_{k_\ell}$ could thus be calculated exactly by some convolution arguments. However, for our purpose it suffices to establish a very crude bound. Observe from \eqref{sum_r} that  there is at least one value among $j_1, \dots, j_\ell$ that is non-zero. Let us assume, without loss of generality, that $j_\ell \neq 0$. We split off the corresponding
random variable $Y_{k_\ell}$ in the indicator in \eqref{form_prod_int_2a}, which is
independent of $Y_{k_1}, \dots, Y_{k_{\ell-1}}$ since $k_1 >\dots > k_\ell$, and
  use the fact that by assumption $Y_{k_\ell}$ is uniformly distributed among the	$2^{k_\ell^{2/3}} + 1$ different values
  in the set $\discset_{k_\ell}$ defined in \eqref{def-discset_recalled}, 
  to obtain
\[ 
\begin{array}{ll}
& \py (j_1 Y_{k_1} + s_2 j_2 Y_{k_2} + \dots + s_\ell j_\ell Y_{k_\ell} = 0) \\
 & \qquad =  \displaystyle\sum_{a \in \mathbb{Z}} \left( \py \big(j_1 Y_{k_1} + s_2 j_2 Y_{k_2} + \dots + s_{\ell-1} j_{\ell-1} Y_{k_{\ell-1}} = a \big) \underbrace{ \py \big(s_\ell j_\ell Y_{k_\ell} = -a \big)}_{\leq 2^{-k_\ell^{2/3}}} \right) \\
& \qquad \leq  2^{-k_\ell^{2/3}} \underbrace{\sum_{a \in \mathbb{Z}} \py \big(j_1 Y_{k_1} + s_2 j_2 Y_{k_2} + \dots + s_{\ell-1} j_{\ell-1} Y_{k_{\ell-1}} = a \big)}_{=1} \\
& \qquad \leq  2^{-\delta_{i,\min}^{2/3}},
\end{array}
\]
where the last inequality holds because $(k_1,\dots,k_\ell)\in\Delta_i$, and  $\delta_{i,\min}$ is by  definition the smallest  element of $\Delta_i$. In the quadruple sum in line \eqref{form_prod_int_2a} the total number of summands is at most $\sqrt{n} \sqrt{n}^{\sqrt{n}} (d+1)^{\sqrt{n}} 2^{\sqrt{n}}$. Note that by construction $\delta_{i,\min} \geq n^{9/10}$ for all $i \geq n^{2/5}$, so that $2^{-\delta_{i,\min}^{2/3}} \leq 2^{(-n^{3/5})}$ for all $i \geq n^{2/5}$. Thus, by a union bound the $\py$-probability that there exists at least one configuration of $\ell, ~(k_1, \dots, k_\ell),~(j_1, \dots, j_\ell),~(s_1, \dots, s_\ell)$ such that $j_1 Y_{k_1} + s_2 j_2 Y_{k_2} + \dots + s_\ell j_\ell Y_{k_\ell} = 0$ holds is bounded above by
$$
\sqrt{n} \sqrt{n}^{\sqrt{n}} (d+1)^{\sqrt{n}} 2^{\sqrt{n}} 2^{(-n^{3/5})}. 
$$
Observe that, since $d$ is fixed, for sufficiently large $n\in\N$, $(d+1) \leq \sqrt{n}$ and 
\[
\sqrt{n} \sqrt{n}^{\sqrt{n}} (d+1)^{\sqrt{n}} 2^{\sqrt{n}} 2^{(-n^{3/5})} \leq 2^{2\sqrt{n}\log_2(n)-n^{3/5}}
\]
for which we can give the crude upper bound $n^{-2}$ holding for all large enough $n\in\N$. Thus,
$$
\py \left(  c_0^{(i)}(Y) \neq b_0^{\sqrt{n}} \right) \leq \frac{1}{n^2},
\quad i \in \{n^{2/5}, \dots, M_n\}, 
$$
for  all sufficiently large $n\in\N$. Now, by  \eqref{def-integn}, $|\{n^{2/5}, \dots, M_n\}| \leq M_n\leq \sqrt{n}$. Thus, we have
$$
\py \left(  \bigcup_{i=n^{2/5}}^{M_n} \left\{c_0^{(i)}(Y) \neq b_0^{\sqrt{n}} \right\} \right) \leq \frac{1}{n^{3/2}},
$$
for all large enough $n\in\N$, which implies the statement of the lemma.   
  \end{proof}

\noindent 
{\em Step 5:  Complete the proof of the LDP stated in Theorem \ref{th4}.}  By the definition of $\integ_n$ in \eqref{def-integn}, we have the relation 
\[
\bigcup_{i=1}^{\integ_n+1}(\Delta_i\cup\Delta_i') \supset \{1,\dots,n\} 
\]
and $M_n\leq \sqrt{n}$.
Together with the fact that $\Delta_i, \Delta'_i, i \in \N,$ are all disjoint,  $|\Delta_i|= \sqrt{n}$, and $|\Delta'_i| = n^{2/5}$,    
this implies 
\begin{eqnarray*}
\sum_{i=n^{2/5}}^{\integ_n} \sqrt{n} & \geq & n - \sum_{i=n^{2/5}}^{\integ_n}|\Delta_i'| - \sum_{i=1}^{n^{2/5}-1}(|\Delta_i|+|\Delta_i'|) -(|\Delta_{\integ_n+1}|+|\Delta_{\integ_n+1}'|) \\
& \geq & n - n^{2/5} \sqrt{n} - 2 \sqrt{n} n^{2/5} - 2 \sqrt{n} \\
& = & n - 3 n^{9/10} - 2 \sqrt{n},
\end{eqnarray*}
while in the other direction trivially $\sum_{i=n^{2/5}}^{\integ_n} \sqrt{n} \leq n$. Thus, for the factor $\prod_{i=n^{2/5}}^{M_n} b_0^{\sqrt{n}}$ appearing in Lemma \ref{lem-estofexpect} we have the lower and upper bounds
\begin{equation}
  \label{b0-est}
  b_0^{n - 3 n^{9/10} - 2 \sqrt{n}} \leq \prod_{i=n^{2/5}}^{M_n} b_0^{\sqrt{n}} \leq b_0^n.
  \end{equation}

Thus,  for any fixed $\theta \in \R$, given any  $\ve > 0$, choosing $d= d(\ve)$ such that \eqref{thus_2} of Lemma \ref{lem-thus1} holds with $p=p_{d(\ve)}$,  then  invoking \eqref{thus_1} as well
as \eqref{indep_prod} of Proposition \ref{prop-interchange}, next 
applying Lemma \ref{lem-estofexpect} with $d = d(\ve)$, $b_0  :=  b_0(\theta; d(\ve))$, and finally
using \eqref{b0-est}  we obtain
\begin{equation} \label{estofexpect-new}
  (1+\ve)^{-n} e^{-5 \theta n^{9/10}} b_0^{n - 3 n^{9/10} - 2 \sqrt{n}} \leq \int_0^1 e^{\theta S_n^Y (\omega)}  d\omega
  \leq b_0^{n} e^{5 \theta n^{9/10}} (1-\ve)^{-n}
\end{equation}
with $\py$-probability at least 
 $1 - n^{-3/2}$,  
 for all sufficiently large $n\in\N$.  Next, note that 
we have $\log (1+\ve) \leq \ve$, and we can (and will) assume that $\ve>0$ is so small that $\log  (1-\ve) \geq -2 \ve$. We also have the trivial bound $5 \theta n^{9/10} + (3 n^{9/10} +2\sqrt{n})\log b_0 = n^{9/10}(5\theta + 3\log b_0)+(2\sqrt{n})\log b_0 \leq n \ve$ for all sufficiently large $n\in \N$. Thus, from \eqref{estofexpect-new} we can deduce that for sufficiently large $n\in\N$, with $\py$-probability at least $1-n^{-3/2}$, 
\begin{align*}
-2\varepsilon \leq -\log (1+\ve) -\frac{1}{n}\Big(5\theta n^{9/10} + (3n^{9/10}+2\sqrt{n})\log b_0\Big) 
& \leq \frac{1}{n} \log \left(\int_0^1 e^{\theta S_n^Y (\omega)}  d\omega \right) - \log b_0 \\
&\leq \frac{1}{n}[5\theta n^{9/10}] - \log (1-\ve)  \\
&\leq 3\ve. 
\end{align*}
This implies  that for all sufficiently large $n\in\N$, 
$$
\py \left( \left| \frac{1}{n} \log \left(\int_0^1 e^{\theta S_n^Y (\omega)}  d\omega \right) -\log b_0 \right| \leq 3 \varepsilon \right) \geq 1 - n^{-3/2}. 
$$
 By the Borel-Cantelli lemma, with $\py$-probability equal to one only finitely many exceptional events occur. This implies that $\py$-almost surely we have
\begin{equation} \label{ublimsup}
\limsup_{n \to \infty} \left| \frac{1}{n} \log \left(\int_0^1 e^{\theta S_n^Y (\omega)}  d\omega \right) -\log b_0 \right| \leq 3 \varepsilon.
\end{equation}
Recall from \eqref{eq:b_0_theta} that $b_0(\theta; d(\ve))$ is a finite polynomial approximation to the modified Bessel function $B_0(\theta)$, the moment generating function defined in \eqref{arcsin_mgf}, and that $b_0(\theta; d(\ve)))$ can be made arbitrarily close to $B_0(\theta)$ by choosing the degree $d=d(\ve)$ sufficiently large.  
Thus,  letting $\ve\to 0$ (and hence $d(\ve)\to+\infty$) and using \eqref{ublimsup} together with \eqref{arcsin_cumgf}, we derive, for every fixed 
$\theta \in \mathbb{R}$,
$$
\lim_{n \to \infty} \frac{1}{n} \log \left( \int_0^1 e^{\theta \sum_{k=1}^n \cos (2 \pi a_k \omega)} d\omega \right) = \log B_0(\theta) = \widetilde{\Lambda} (\theta) \qquad \text{$\py$-a.s.}
$$ 
Since $\widetilde{\Lambda} (\theta)$ is a continuous (in fact, differentiable) function in $\theta$, we can deduce that $\py$-almost surely this result holds for all $\theta \in \mathbb{R}$: for $\py$-almost all realizations of the random sequence $Y$, or equivalently, $a_1^Y, a_2^Y, \dots$, we have
$$
\lim_{n \to \infty} \frac{1}{n} \log \left( \int_0^1 e^{\theta S_n^Y(\omega)} d\omega \right) =
 \lim_{n \to \infty} \frac{1}{n} \log \left( \int_0^1 e^{\theta \sum_{k=1}^n \cos (2 \pi a_k^Y \omega)} d\omega \right) = \widetilde{\Lambda}(\theta) \qquad \text{for all $\theta \in \mathbb{R}$.}
$$
Together with the G\"artner-Ellis theorem, Theorem \ref{th-GE},  this proves the desired result.

\subsection*{Acknowledgement}
CA is supported by the Austrian Science Fund (FWF), projects F-5512, I-3466, I-4945 and Y-901. ZK is supported by the German Research Foundation under Germany's Excellence Strategy EXC 2044 -- 390685587, Mathematics M\"unster: Dynamics - Geometry - Structure. JP is supported by the Austrian Science Fund (FWF), projects P32405 and the Special Research Program F5508-N26. KR is supported by the National Science Foundation (NSF) Grant DMS-1954351 and the Roland George Dwight Richardson Chair at Brown University. 
We also gratefully acknowledge the support of the Oberwolfach Research Institute for Mathematics, where initial discussions were held during the workshop ``New Perspectives and Computational Challenges in High Dimensions'' (Workshop ID 2006b).

\appendix 
  
\section{Proof of Proposition~\ref{prop:I_2_I_3}}
\label{ap-a}

Fix an integer  $q\in \{2,3,\ldots\}$. For $m\in\N$ and $n\in\N$ recall that $A_m(n)$ denotes the number of solutions to the equation
\begin{equation}\label{eq:sum_powers_two}
\sum_{i=1}^m \eps_i q^{k_i} = 0
\end{equation}
in the unknowns $k_1,\ldots,k_m\in \{1,\ldots,n\}$ and $\eps_1,\ldots,\eps_m \in \{+1,-1\}$.
\begin{prop}\label{prop:a_m_n}
Fix $m\in\N$. Then, the function $A_m(n)$ restricted to the values $n\geq m-2$ is a polynomial in $n$ of degree at most $[m/2]$.
\end{prop}
\begin{proof}
  Let $A_{m,p_1,p_2}(n)$ be the number of representations of zero as a sum of signed powers of $2$ which begins  with $p_1$ terms of the form $+q^1$ followed by  $p_2$ terms of the form $-q^1$ and does not contain any 
  more $\pm q$-terms. More precisely, for $p_1,p_2\in\N_0$ such that $p_1+p_2\leq m$, we define  $A_{m,p_1,p_2}(n)$ to be the number of solutions to~\eqref{eq:sum_powers_two} such that
\begin{align*}
&k_1=\ldots=k_{p_1+p_2} = 1,
\\
&\eps_1=\ldots=\eps_{p_1} = +1,
\\
&\eps_{p_1+1}=\ldots=\eps_{p_1+p_2} = -1,
\\
&k_i\in \{2,\ldots,n\} \text{ for } i\in \{p_1+p_2+1,\ldots,n\}.
\end{align*}
Since in any general solution to~\eqref{eq:sum_powers_two} the terms $\pm q$ can appear at arbitrary positions, we have
$$
A_m(n) = \sum_{\substack{p_1,p_2\geq 0\\p_1+p_2\leq m}} \binom {m}{p_1+p_2} \binom{p_1+p_2}{p_1} A_{m,p_1,p_2}(n).
$$
To establish Proposition~\ref{prop:a_m_n} it suffices to prove the following two claims for all $\ell\in\N$:
\begin{enumerate}
\item[(a)] $A_{\ell,0,0}(n)$ is a polynomial in $n$ of degree at most $[\ell/2]$ in the range $n\geq \ell-2$.
\item[(b)] For $(p_1,p_2)\neq (0,0)$, $A_{\ell,p_1,p_2}(n)$ is a polynomial in $n$ of degree at most $[\ell/2]-1$ in the range $n\geq \ell-3$.
\end{enumerate}
First of all, observe that these claims are true for $\ell=1$ and $\ell=2$ because
$$
A_{1,0,0}(n) =0,
\quad
A_{2,0,0}(n) = 2n-2,
\quad
A_{2,1,1}(n) = 2,
\quad
A_{2,0,1}(n)=A_{2,1,0}(n) = 0.
$$
For larger values of $\ell$, we shall prove these claims by induction. The inductive argument is based on certain recurrence relations for the functions $A_{m,p_1,p_2}(n)$ that we now derive. 

\vspace*{2mm}
\noindent
\textit{Case 1.}
Let first $p_1=p_2 = p\in \N_0$. Then, in~\eqref{eq:sum_powers_two} we can cancel the $+q$-terms with the $-q$-terms, which yields a representation of $0$ as a sum of $\pm q^2, \pm q^3,\ldots, \pm q^n$, the total number of terms being $m-p_1-p_2$. Dividing all terms by $q$, we obtain a representation of $0$ as a sum of $m-p_1-p_2$ terms of the form $\pm q, \pm q^2,\ldots, \pm q^{n-1}$. The number of such representations is $A_{m-p_1-p_2}(n-1)$. Hence, we arrive at
\begin{equation}\label{eq:rec_A_case1}
A_{m, p, p}(n) = A_{m-2p}(n-1) = \sum_{\substack{r_1,r_2\geq 0\\ r_1+r_2\leq m-2p}} \binom{m-2p}{r_1+r_2}\binom{r_1+r_2}{r_1} A_{m-2p,r_1,r_2}(n-1).
\end{equation}

\vspace*{2mm}
\noindent
\textit{Case 2.}
Let now $p_1>p_2$. Then, in the representation~\eqref{eq:sum_powers_two} we can cancel $p_2$ terms of the  form $+q^1$ with the same number $p_2$ of terms of the form $-q^1$. The resulting representation of $0$ contains $p_1-p_2>0$ terms of the form $+q^1$ and $m-p_1- p_2$ terms of the form $\pm q^2,\pm q^3,\ldots, \pm q^{n}$. If $p_1-p_2$ is not divisible by $q$, then $A_{m,p_1,p_2}(n) = 0$ because the sum on the left-hand side of~\eqref{eq:sum_powers_two} is not divisible by $q^2$. So, assume that $p_1-p_2 = s q$ for some $s\in\N$. Divide the remaining $p_1-p_2$ terms of the form $+q^1$ into $s$ groups of the form $+q^1+\ldots+ q^1$, each consisting of $q$ terms, and replace each group by $+q^2$. We obtain $s$ terms of the form $q^2$. However, we have also to take care of the terms of the form $\pm q^2$ that can  appear among the $m-p_1-p_2$ terms of the form $\pm q^2, \pm q^3,\ldots, \pm q^n$. Let $r_1$, respectively, $r_2$,  be the number of the terms $+q^2$, respectively, $-q^2$, among these $m-p_1-p_2$ terms. Dividing all terms by $q$, we obtain a representation of $0$ starting with $s = (p_1-p_2)/q$ terms of the form $+q^1$, followed by a sum of $m-p_1-p_2$ terms of the form $\pm q^1,\pm q^2,\ldots, \pm q^{n-1}$, among which $r_1$ terms are of the form $+q^1$ and $r_2$ terms are of the form $-q^1$. Since the positions of these terms can be arbitrary among the $m-p_1-p_2$ terms, we arrive at the identity
\begin{equation}\label{eq:rec_A_case2}
A_{m,p_1,p_2} (n) = \sum_{\substack{r_1,r_2\geq 0\\ r_1+r_2 \leq m-p_1 -p_2}} \binom{m-p_1 -p_2}{r_1+r_2}\binom {r_1+r_2}{r_1} A_{s + m-p_1-p_2, s+r_1, r_2}(n-1),
\end{equation}
which holds if $p_1-p_2 = sq$ for $s\in\N$.

\vspace*{2mm}
\noindent
\textit{Case 3.}
Similar arguments show that in the case when $p_1<p_2$ we have $A_{m,p_1,p_2}(n) = 0$ if $p_2-p_1$ is not divisible by $q$ and
\begin{equation}\label{eq:rec_A_case3}
A_{m,p_1,p_2} (n) = \sum_{\substack{r_1,r_2\geq 0\\ r_1+r_2 \leq m-p_1 -p_2}} \binom{m-p_1 -p_2}{r_1+r_2}\binom {r_1+r_2}{r_1} A_{s + m-p_1 -p_2, r_1, s + r_2}(n-1),
\end{equation}
if $p_2-p_1 = sq$ for some $s\in\N$.

\vspace*{2mm}
We are now in position to prove claims (a) and (b) by induction. As already mentioned above, the claims are true for $\ell=1,2$. Assume that the claims are true for $\ell= 1,\ldots,m-1$ with some $m\in \{3,4,\ldots\}$. We prove them for $\ell=m$.

\vspace*{2mm}
\noindent
\textit{Case A.} Consider first the case when $(p_1,p_2)\neq (0,0)$. Then, \eqref{eq:rec_A_case1}, \eqref{eq:rec_A_case2}, \eqref{eq:rec_A_case3} yield a representation of $A_{m,p_1,p_2}(n)$ as a linear combination of the terms $A_{\ell, r_1,r_2}(n-1)$ with $\ell<m$. Applying the induction assumption, we obtain that $A_{m,p_1,p_2}(n)$ is a polynomial in $n$ of degree at most $[m/2]-1$ in the range $n\geq m-3$.
In the individual cases, this can be seen as follows:
\begin{itemize}
\item Case 1: If $p_1=p_2=p\neq 0$, then from~\eqref{eq:rec_A_case1} we have $\ell = m - 2p < m$.  By the induction assumptions (a) and (b), the terms $A_{m-2p, r_1,r_2}(n-1)$ appearing in~\eqref{eq:rec_A_case1} are polynomials in $(n-1)$ of degree at most $[(m-2p)/2]\leq [m/2]-1$ in the range $n-1\geq m-2p-2$, which lies in  the range $n \geq m-3$.
\item Case 2: If $p_1>p_2$ and $p_1-p_2 = sq$ for $s\in\N$, then $\ell = s + m - p_1 - p_2$, which is
  strictly less than $m$ since $q \geq 2$. By the induction assumption (b), the terms $A_{s + m-p_1-p_2, s+r_1, r_2}(n-1)$  (for which we have $s+r_1>0$ since $s\in\N$) appearing in~\eqref{eq:rec_A_case2} are polynomials of $(n-1)$ of degree at most $[(s+m-p_1 - p_2)/2]-1\leq [m/2]-1$ in the range $n-1\geq s+m-p_1-p_2-3$. This lies
  in  the range $n\geq m-3$ since
  $p_1+p_2 - s = p + p_2 - (p_1-p_2)/q>0$ and hence, being integral, is greater than or equal to $1$. 
\item Case 3: If $p_2>p_1$ and $p_2-p_1 = sq$ for $s\in\N$, then $\ell = s + m - p_1 - p_2 < m$. The remaining considerations are similar to Case 2.
\end{itemize}
In all three cases we obtain that (b) holds for $\ell = m$.

\vspace*{2mm}
\noindent
\textit{Case B.}
Consider now the case when $p_1= p_2 = 0$. Then,~\eqref{eq:rec_A_case1} yields
$$
A_{m, 0, 0}(n) = A_{m}(n-1) = \sum_{\substack{r_1,r_2\geq 0\\ r_1+r_2\leq m}} \binom{m-2p}{r_1+r_2}\binom{r_1+r_2}{r_1} A_{m,r_1,r_2}(n-1).
$$
Separating the term with $(r_1,r_2) = (0,0)$, we obtain
$$
A_{m, 0, 0}(n) = A_{m,0,0}(n-1) + \sum_{\substack{r_1,r_2\geq 0\\ r_1+r_2\leq m\\ (r_1,r_2) \neq (0,0)}} \binom{m-2p}{r_1+r_2}\binom{r_1+r_2}{r_1} A_{m,r_1,r_2}(n-1).
$$
To each term in the sum on the right-hand side we can apply the same considerations as in Case A, due to
the restriction $(r_1,r_2)\neq (0,0)$. Thus, the sum on the right-hand side is a polynomial in $n$ of degree at most  $[m/2]-1$ in the range $n-1\geq m-3$. Denoting this polynomial by $P_m(n)$, we have
$$
A_{m, 0, 0}(n) = A_{m,0,0}(n-1) + P_m(n)
$$
for all $n\geq m-2$. Iterating this, we obtain
$$
A_{m,0,0}(n) = P_m(n) + P_m(n-1) + \ldots + P_m(m-2) + A_{m,0,0}(m-3),
$$
for all $n\geq m-2$. The right-hand side is a polynomial in $n$ of degree at most $[m/2]$. This proves that (a) holds with $\ell = m$, thus completing the induction.
\end{proof}
Proposition~\ref{prop:a_m_n} allows us to find explicit formulae for $A_m(n)$ for every fixed $m$ and all $n\geq m-2$. This also yields the moments of the lacunary sums $S_n$ because, as shown in Lemma \ref{lem:A_m_n_moment},  these are given by
$$
\E[S_n^m] = \frac{A_m(n)}{2^m}, \qquad m,n\in\N.
$$
To compute $A_m(n)$, we can proceed as follows. Let some $m\in \N$ be given. Using computer algebra, calculate the values $A_{m}(n)$ for  $n= m-2,\ldots, m-2+[m/2]$. For example, one may just expand the Laurent polynomial
$$
\left(\sum_{k=1}^n \left(x^{+q^k} + x^{-q^k}\right)  \right)^m
$$
and observe that $A_m(n)$ is the coefficient of $x^0$ there.
Then, compute the unique interpolating polynomial of degree $[m/2]$ taking the same values as $A_m(n)$ for $n= m-2,\ldots, m-2+[m/2]$. By Proposition~\ref{prop:a_m_n}, this yields a formula for $A_m(n)$ for all $n\geq m-2$.
For example, for $q=2$ we obtained the following formula
\begin{align*}
&A_1(n) = 0 &\text{ for all } n\in \N,\\
&A_2(n) = 2n &\text{ for all } n\in \N,\\
&A_3(n) = 6n-6 &\text{ for all } n\in \N,\\
&A_4(n) = 12 n^2 + 18 n - 48 &\text{ for all } n\geq 2,\\
&A_5(n) = 120 n^2 -130 n - 240  &\text{ for all } n\geq 3,\\
&A_6(n) = 120 n^3 + 900  n^2 - 3310 n + 870   &\text{ for all } n\geq 4,\\
&A_7(n) = 2520 n^3 + 840  n^2 - 40446 n + 48552   &\text{ for all } n\geq 5,
\end{align*}
and so on.  By computing more values of $A_m(n)$ than necessary, it is also possible to check the correctness of these formulas.  
Since the $m$-th cumulant $\kappa_m(S_n)$ of $S_n$ can be expressed as a polynomial of the first $m$ moments $\E [S_n], \ldots, \E [S_n^m]$, we obtain that $\kappa_m(S_n)$ is a polynomial in $n$ of degree at most $[m/2]$ for all $n\geq m-2$. In fact, it is even a polynomial of degree $1$. To see this, recall that the convergence of analytic functions in~\eqref{eq:cum_gen_funct_converge} is uniform on some disk around $0$. Differentiating~\eqref{eq:cum_gen_funct_converge} $m\in\N$ times, we get
$$
\lim_{n\to\infty}\frac 1n \kappa_m(S_n) = \Lambda_q^{(m)}(0),
$$
which implies that $\kappa_m(S_n)$ must be of degree $1$. 
For example, in the case when $q=2$, we obtained
\begin{align*}
&\frac{\kappa_1(S_n)}{1!} = 0 &\text{ for all } n\in \N,\\
&\frac{\kappa_2(S_n)}{2!} = \frac n4 &\text{ for all } n\in \N,\\
&\frac{\kappa_3(S_n)}{3!} = \frac{n-1}{8} &\text{ for all } n\in \N,\\
&\frac{\kappa_4(S_n)}{4!} = \frac{3n-8}{64} &\text{ for all } n\geq 2,\\
&\frac{\kappa_5(S_n)}{5!} = \frac{-n-24}{384} &\text{ for all } n\geq 3,\\
&\frac{\kappa_6(S_n)}{6!} = \frac{-115 n - 51}{4608}  &\text{ for all } n\geq 4,\\
&\frac{\kappa_7(S_n)}{7!} = \frac{916 - 393 n}{15360}  &\text{ for all } n\geq 5,
\end{align*}
and so on. This yields the first few terms in the Taylor expansion of $\Lambda_2$. Since $I_2'$ is the inverse function of $\Lambda_2'$, this easily yields the Taylor expansion of $I_2$ stated in Proposition~\ref{prop:I_2_I_3}.

\bibliography{Lacunary_LDP}
\bibliographystyle{abbrv}

\end{document}